\documentclass[11pt,leqno]{amsart}%
\usepackage{amsmath}
\usepackage{amsfonts}
\usepackage{amssymb}
\usepackage[usenames]{color}
\usepackage{graphicx}%
\providecommand{\U}[1]{\protect\rule{.1in}{.1in}}
\providecommand{\U}[1]{\protect\rule{.1in}{.1in}}
\newtheorem{theorem}{Theorem}[section]

\newtheorem{corollary}[theorem]{Corollary}

\newtheorem{definition}[theorem]{Definition}

\newtheorem{lemma}[theorem]{Lemma}

\newtheorem{proposition}[theorem]{Proposition}
\newtheorem{remark}[theorem]{Remark}

\begin{document}
\title[Global maximum principles and divergence theorems]{Global maximum principles and divergence theorems on complete manifolds with boundary}
\author{Debora Impera}
\address{Dipartimento di Matematica e Applicazioni\\
Universit\`a di Milano--Bicocca Via Cozzi 53\\
I-20125 Milano, ITALY}
\email{debora.impera@unimib.it}
\author{Stefano Pigola}
\address{Dipartimento di Scienza e Alta Tecnologia\\
Universit\`a dell'Insubria - Como\\
via Valleggio 11\\
I-22100 Como, ITALY}
\email{stefano.pigola@uninsubria.it}
\author{Alberto G. Setti}
\address{Dipartimento di Scienza e Alta Tecnologia \\
Universit\`a dell'Insubria - Como\\
via Valleggio 11\\
I-22100 Como, ITALY}
\email{alberto.setti@uninsubria.it}
\maketitle

\begin{abstract}
In this paper we extend to non-compact Riemannian manifolds with boundary the
use of two important tools in the geometric analysis of compact spaces,
namely, the weak maximum principle for subharmonic functions and the
integration by parts. The first one is a new form of the classical Ahlfors
maximum principle whereas the second one is a version for manifolds with
boundary of the so called Kelvin-Nevanlinna-Royden criterion of parabolicity.
In fact, we will show that the validity of non-compact versions of these tools
serve as a characterization of the Neumann parabolicity of the space.

The motivation underlying this study is to obtain new information on the
geometry of graphs with prescribed mean curvature inside a Riemannian product
of the type $N\times\mathbb{R}$. In this direction two kind of results will be
presented: height estimates for constant mean curvature graphs parametrized
over unbounded domains in a complete manifold and slice type results for
graphs whose superlevel sets have finite volume.

\end{abstract}
\tableofcontents

\section*{Introduction}

This paper aims at extending to non-compact Riemannian manifolds with boundary
the use of two important tools in the geometric analysis of compact spaces,
namely, the integration by parts and the weak maximum principle for
subharmonic functions. The motivation underlying this study is mainly the
attempt to obtain new information on the geometry of graphs or, more
generally, of hypersurfaces with boundary and prescribed mean curvature inside
a Riemannian product of the type $N\times\mathbb{R}$.\bigskip

In the setting of Riemannian manifolds without boundary, it is by now well
known that parabolicity represents a good substitute of the compactness of the
underlying space, see e.g. the account in \cite{PS-Fortaleza}. Thus, in order
to extend the use of the classical tools alluded to above, we are naturally to
a deeper study of parabolicity for manifolds with boundary. As we shall see in
Appendix \ref{appendix-different} there are several concepts of parabolicity
in this setting and they are in a certain hierarchy, so one has to make a
choice. In view of our geometric purposes we decided to follow the more
traditional path, \cite{Gr, Gr1, Gr2, GN}, that, from the stochastic
viewpoint, translates into the property that the reflected Brownian motion be
recurrent. This is the strongest of the notions of parabolicity known \ in the
literature, but it is also the one which seems to be more related to the
geometry of the space. Thus, for instance, every proper minimal graph over a
smooth domain of $\mathbb{R}^{2}$ is parabolic in our traditional sense
because of its area growth property; see Appendix \ref{appendix-different}. In
order to put the precise definition of parabolicity we need to recall the
notion of weak sub (super) solution subjected to Neumann boundary
conditions.\bigskip

Let $\left(  M,g\right)  $ be an oriented Riemannian manifold with smooth
boundary $\partial M\neq\emptyset$ and exterior unit normal $\nu.$ By a domain
in $M$ we mean a non-necessarily connected open set $D\subseteq M$. We say
that the domain $D$ is smooth if its topological boundary $\partial D$ is a
smooth hypersurface $\Gamma$ with boundary $\partial\Gamma=\partial
D\cap\partial M$. Clearly, if $\partial M=\emptyset$ then the smoothness
condition reduces to the usual one. It is a standard fact that every manifold
$M$ with (possibly empty) boundary has an exhaustion by smooth pre-compact
domains. Simply choose a proper smooth function $\rho:M\rightarrow
\mathbb{R}_{\geq0}$ and, according to Sard theorem, take a sequence $\left\{
t_{k}\right\}  \nearrow+\infty$ such that $t_{k}$ is a regular value for both
$\left.  \rho\right\vert _{\mathrm{int}M}$ and $\left.  \rho\right\vert
_{\partial M}$. Then $D_{k}=\left\{  \rho<t_{k}\right\}  $ defines the desired
exhaustion with smooth boundary $\partial D_{k}=\left\{  \rho=t_{k}\right\}  $.

Adopting a notation similar to the one in \cite{Gr1}, for any domain
$D\subseteq M$ we define%
\[
\partial_{0}D=\partial D\cap\mathrm{int}M.
\]
Note also that $D$ could include part of the boundary of $M$. We therefore set%
\[
\partial_{1}D=\partial M\cap D
\]
\bigskip

Now, suppose $D\subseteq M$ is any domain. We put the following

\begin{definition}
By a weak Neumann solution $u\in W_{loc}^{1,2}\left(  D\right)  $ of the
problem%
\begin{equation}
\left\{
\begin{array}
[c]{ll}%
\Delta u\geq0 & \text{on }D\\
\dfrac{\partial u}{\partial\nu}\leq0 & \text{on }\partial_{1}D,
\end{array}
\right.  \label{subneumannproblem}%
\end{equation}
we mean that the following inequality%
\begin{equation}
-\int_{D}\left\langle \nabla u,\nabla\varphi\right\rangle \geq0 \label{subsol}%
\end{equation}
holds for every $0\leq\varphi\in C_{c}^{\infty}\left(  D\right)  $. Similarly,
by taking $D=M$, one defines the notion of weak Neumann subsolution of the
Laplace equation on $M$ as a function $u\in W_{loc}^{1,2}\left(  M\right)  $
which satisfies (\ref{subsol}) for every $0\leq\varphi\in C_{c}^{\infty
}\left(  M\right)  $. \ As usual, the notions of weak supersolution and weak
solution can be obtained by reversing the inequality or by replacing the
inequality with an equality in (\ref{subsol}), and removing the sign condition
on $\varphi$.
\end{definition}

\begin{remark}
Clearly, in the above definition, it is equivalent to require that
(\ref{subsol}) holds for every $0\leq\varphi\in Lip_{c}\left(  M\right)  $.
Note also that standard density arguments work even for manifolds with
boundary and, therefore, (\ref{subsol}) extends to all compactly supported
$0\leq\varphi\in W_{0}^{1,2}\left(  D\right)  $. Here, as usual, $W_{0}%
^{1,2}\left(  D\right)  $ denotes the closure of $C_{c}^{\infty}\left(
D\right)  $ with respect to the $W^{1,2}$-norm.
\end{remark}

\begin{remark}
Note that in the equality case we have the usual notion of variational
solution of the mixed problem
\[%
\begin{cases}
\Delta u=0 & \text{on }D\\
\dfrac{\partial u}{\partial\nu}=0 & \text{on }\partial_{1}D\\
u=0 & \text{on }\partial_{0}D.
\end{cases}
\]

\end{remark}

\begin{remark}
If $\partial M=\emptyset$ or, more generally, $D\subseteq\mathrm{int}M$, the
Neumann condition disappears and we recover the usual definition of weak sub-
(super-)solution. Obviously, in the smooth setting, a classical solution of
(\ref{subneumannproblem}) is also a weak Neumann subsolution as one can verify
using integration by parts. Actually, this is true in a more general setting.
See Definition \ref{def_weak-sol-divX} and Lemma \ref{lemma_equiv-weak-def} in
Subsection \ref{subsection-divergence}.
\end{remark}

We are now ready to give the following definition of parabolicity in the form
of a Liouville-type result.

\begin{definition}
\label{def_parab}An oriented Riemannian manifold $M$ with boundary $\partial
M\neq\emptyset$ is said to be parabolic if any bounded above, weak Neumann
subsolution of the Laplace equation on $M$ must be constant. Explicitly, for
every $u\in C^{0}\left(  M\right)  \cap W_{loc}^{1,2}\left(  M\right)  $,%
\begin{equation}
\label{def_par}%
\begin{array}
[c]{ccc}%
\left\{
\begin{array}
[c]{ll}%
\Delta u\geq0 & \text{on }M\\
\dfrac{\partial u}{\partial\nu}\leq0 & \text{on }\partial M\\
\sup_{M}u<+\infty &
\end{array}
\right.  & \Rightarrow & u\equiv\mathrm{const}.
\end{array}
\end{equation}

\end{definition}

It is known from \cite{Gr1}\ that, in case $M$ is complete with respect to the
intrinsic distance function $d$, then geometric conditions implying
parabolicity rely on volume growth properties of the space. In order to give
the precise statement it is convenient to introduce some notation. Having
fixed a reference origin $o\in\mathrm{int}M$, we set $B_{R}^{M}\left(
o\right)  =\left\{  x\in M:d\left(  x,o\right)  <R\right\}  $ and $\partial
B_{R}^{M}\left(  o\right)  =\left\{  x\in M:d\left(  x,o\right)  =R\right\}
$, the metric ball and sphere of $M$ centered at $o$ and of radius $R>0$. We
also denote by $r\left(  x\right)  =d\left(  x,o\right)  $ the distance
function from $o$. Clearly, $r\left(  x\right)  $ is Lipschitz, hence
differentiable a.e. in $\mathrm{int}M$. Moreover, for a.e. $x\in\mathrm{int}%
M$, differentiating $r$ along a minimizing geodesic from $o$ to $x$ (which
exists by completeness) we easily see that the usual Gauss Lemma holds,
namely, $\left\vert \nabla r\right\vert =1$ a.e. in $\mathrm{int}M$.
Therefore, by the co-area formula applied to $\left.  r\right\vert
_{\mathrm{int}M}$ and the fact that $\mathrm{vol}B_{R}^{M}\left(  o\right)
=\mathrm{vol}\left(  B_{R}^{M}\left(  o\right)  \cap\mathrm{int}M\right)  $,
we have%
\[
\frac{d}{dR}\mathrm{vol}B_{R}^{M}\left(  o\right)  =\mathrm{Area}\left(
\partial_{0}B_{R}^{M}\left(  o\right)  \right)  ,
\]
for a.e. $R>0$.

The following result is due to Grigor'yan \cite{Gr1}. For a proof in the
$C^{2}$ case see Theorem \ref{th_areagrowth_meancurvop} and Remark
\ref{rmk_areagrowth_meancurvop}.

\begin{theorem}
\label{th_growth}Let $\left(  M,g\right)  $ be a complete Riemannian manifold
with boundary $\partial M\neq\emptyset$. If, for some reference point $o\in
M$, either%
\[
\frac{R}{\mathrm{vol}B_{R}^{M}\left(  o\right)  }\notin L^{1}\left(
+\infty\right)
\]
or%
\[
\frac{1}{\mathrm{Area}\left(  \partial_{0}B_{R}^{M}\left(  o\right)  \right)
}\notin L^{1}\left(  +\infty\right)
\]
then $M$ is parabolic.
\end{theorem}

It is a usual consequence of the co-area formula that the area growth
condition is weaker than the volume growth condition. On the other hand, the
volume growth condition is more stable with respect to (even rough)
perturbations of the metric and sometimes it characterizes the parabolicity of
the space. Therefore, both are important.

The first main result of the paper is the following maximum principle
characterization of parabolicity. It extends to manifolds with boundary a
classical result by L.V. Ahlfors.

\begin{theorem}
[Ahlfors maximum principle]\label{th_intro_Ahlfors}$M$ is parabolic if and
only the following maximum principle holds. For every domain $D\subseteq M$
with $\partial_{0}D\neq\emptyset$ and for every $u\in C^{0}\left(
\overline{D}\right)  \cap W_{loc}^{1,2}\left(  D\right)  $ satisfying, in the
weak Neumann sense,%
\[
\left\{
\begin{array}
[c]{ll}%
\Delta u\geq0 & \text{on }D\\
\dfrac{\partial u}{\partial\nu}\leq0 & \text{on }\partial_{1}D\\
\sup\limits_{D}u<+\infty, &
\end{array}
\right.
\]
it holds%
\[
\sup_{D}u=\sup_{\partial_{0}D}u.
\]

\end{theorem}

It is worth to observe that, in case $D=M$, the Neumann boundary condition
plays no role and the result takes the following form which is crucial in the applications.

\begin{theorem}
\label{th_ahlfors-wholeM}Let $M$ be a parabolic manifold with boundary
$\partial M\neq\emptyset$. If $u\in C^{0}\left(  M\right)  \cap W_{loc}%
^{1,2}\left(  \mathrm{int}M\right)  $ satisfies%
\[
\left\{
\begin{array}
[c]{ll}%
\Delta u\geq0 & \text{on }\mathrm{int}M\\
\sup_{M}u<+\infty &
\end{array}
\right.
\]
then%
\[
\sup_{M}u=\sup_{\partial M}u.
\]

\end{theorem}

It is not surprising that this global maximum principle proves to be very
useful to get height estimates for constant mean curvature hypersurfaces in
product spaces. By way of example, we point out the following

\begin{theorem}
[Height estimate]\label{th_intro_hest} Let $N$ be a Riemannian manifold
without boundary and Ricci curvature satisfying $Ric_{N}\geq0$. Let $\Sigma$
be a complete, oriented hypersurface in $N\times\mathbb{R}$ with boundary
$\partial\Sigma\neq\emptyset$ and satisfying the following requirements:

\begin{enumerate}
\item[(i)] $\Sigma$ has quadratic intrinsic volume growth%
\begin{equation}
\mathrm{vol}B_{R}^{\Sigma}\left(  o\right)  =O\left(  R^{2}\right)  ,\text{ as
}R\rightarrow+\infty; \label{vg-hesimate}%
\end{equation}

\item[(ii)] $\partial\Sigma$ is contained in the slice $N\times\left\{
0\right\}  $;

\item[(iii)] For a suitable choice of the Gauss map $\mathcal{N}$ of $\Sigma$,
the hypersurface $\Sigma$ has constant mean curvature $H>0$ and the angle
$\Theta$ between $\mathcal{N}$ and the vertical vector field $\partial
/\partial t$ is contained in the interval $[\frac{\pi}{2},\frac{3\pi}{2}]$,
i.e.,%
\[
\cos\Theta=\left\langle \mathcal{N},\frac{\partial}{\partial t}\right\rangle
\leq0.
\]

\end{enumerate}

If $\Sigma$ is contained in a slab $N\times\lbrack-T,T]$ for some $T>0$, then%
\[
\Sigma\subseteq N\times\left[  0,\frac{1}{H}\right]  .
\]

\end{theorem}

We observe explicitly that (\ref{vg-hesimate}) can be replaced by the stronger
extrinsic condition%
\[
\mathrm{vol}\left(  B_{R}^{N}\left(  o\right)  \cap\Sigma\right)  =O\left(
R^{2}\right)  ,\text{ as }R\rightarrow+\infty,
\]
which, in turn, follows from the relation%
\[
B_{R}^{\Sigma}\left(  o\right)  \subseteq B_{R}^{N}\left(  o\right)
\cap\Sigma.
\]
We also note that there are important situations where the assumption on the
Gauss map is automatically satisfied and the volume growth condition on the
hypersurface is inherited from that of the ambient space. The following height
estimate extends previous results for $H$-graphs over non-compact domains
(\cite{He}, \cite{HoLiRo}, \cite{CheRo}, \cite{Sp}).

\begin{theorem}
[Height estimate for graphs]\label{th_intro_hest_graph}Let $\left(
N,g\right)  $ be a complete, Riemannian manifold without boundary satisfying
$Ric_{N}\geq0$ and%
\[
\mathrm{vol}B_{R}^{N}\left(  o\right)  =O\left(  R^{2}\right)  ,\text{ as
}R\rightarrow+\infty.
\]
Let $M\subset N$ be a closed domain with smooth boundary $\partial
M\neq\emptyset$. Suppose we are given a graph $\Sigma$ over $M$ with boundary
$\partial\Sigma\subset M\times\left\{  0\right\}  $ and constant mean
curvature $H>0$ with respect to the downward Gauss map. If $\Sigma$ is
contained in a slab, then%
\[
\Sigma\subseteq M\times\left[  0,\frac{1}{H}\right]  .
\]

\end{theorem}

In the particular case of graphs over a domain of a surface of non-negative
Gauss curvature we obtain the following result that extends to non-homogeneous
surfaces Theorem 4 in \cite{RoRo}.

\begin{corollary}
\label{coro_intro_hest_graph} Let $\left(  N,g\right)  $ be a complete
$2$-dimensional Riemannian manifold without boundary of non-negative Gauss
curvature. Let $M\subset N$ be a closed domain with smooth boundary $\partial
M\neq\emptyset$. Suppose we are given a graph $\Sigma$ over $M$ with boundary
$\partial\Sigma\subset M\times\left\{  0\right\}  $ and constant mean
curvature $H>0$ with respect to the downward Gauss map. Then%
\[
\Sigma\subseteq M\times\left[  0,\frac{1}{H}\right]  .
\]

\end{corollary}

In the setting of manifolds without boundary, it is well known from a
classical work by T. Lyons and D. Sullivan \cite{LS} that the validity of an
$L^{2}$-divergence theorem is related, and in fact equivalent, to the
parabolicity of the space. We shall complete the picture by extending the
$L^{2}$-divergence theorem to non-compact manifolds with boundary.

\begin{theorem}
[$L^{2}$-divergence theorem]\label{th_intro_Stokes}Let $M$ be a parabolic
Riemannian manifold with boundary $\partial M\neq\emptyset$ and outward
pointing unit normal $\nu$. Then $M$ is parabolic if and only if the following
holds. Let $X$ be a vector field on $M$ satisfying the following conditions:%
\begin{align}
&  \text{(a) }\left\vert X\right\vert \in L^{2}\left(  M\right) \label{KNR1}\\
&  \text{(b) }\left\langle X,\nu\right\rangle \in L^{1}\left(  \partial
M\right) \nonumber\\
&  \text{(c) }\operatorname{div}X\in L_{loc}^{1}(M),\ \left(
\operatorname{div}X\right)  _{-}\in L^{1}\left(  M\right)  .\nonumber
\end{align}
Then%
\[
\int_{M}\operatorname{div}X=\int_{\partial M}\left\langle X,\nu\right\rangle
.
\]

\end{theorem}

A weaker version of the $L^{2}$-divergence theorem, involving solutions $X$ of
inequalities of the type $\operatorname{div}X\geq f$ with boundary conditions
$\left\langle X,\nu\right\rangle \leq0$, will be employed in our
investigations on hypersurfaces in product spaces; see Proposition
\ref{propineq}. In particular, from this latter we shall obtain the following
result for hypersurfaces contained in a half-space of $N\times\mathbb{R}$.

\begin{theorem}
[Slice theorem]\label{th_intro_slice} Let $N$ be a Riemannian manifold without
boundary. Let $\Sigma\subset N\times\lbrack0,+\infty)$ be a complete, oriented
hypersurface with boundary $\partial\Sigma\neq\emptyset$ contained in the
slice $N\times\left\{  0\right\}  $ and satisfying the volume growth condition%
\[
\mathrm{vol}B_{R}^{\Sigma}\left(  o\right)  =O\left(  R^{2}\right)  ,\text{ as
}R\rightarrow+\infty.
\]
Assume that, for a suitable choice of the Gauss map $\mathcal{N}$ of $\Sigma$,
the hypersurface $\Sigma$ has non-positive mean curvature $H\left(  x\right)
\leq0$ and the angle $\Theta$ between $\mathcal{N}$ and the vertical vector
field $\partial/\partial t$ is contained in the interval $[\frac{\pi}{2}%
,\frac{3\pi}{2}]$, i.e.,%
\[
\cos\Theta=\left\langle \mathcal{N},\frac{\partial}{\partial t}\right\rangle
\leq0.
\]
If there exists some half-space $N\times\lbrack t,+\infty)$ of $N\times
\mathbb{R}$ such that%
\[
\mathrm{vol}\left(  \Sigma\cap N\times\lbrack t,+\infty)\right)  <+\infty,
\]
then $\Sigma\subset N\times\left\{  0\right\}  $.
\end{theorem}

In case $\Sigma$ is given graphically over a parabolic manifold $M$, we shall
obtain the following variant of the slice theorem that involves the volumes of
orthogonal projections of $\Sigma$ on $M.$ Its proof \ requires a
Liouville-type theorem for the mean curvature operator under volume growth
conditions; see Theorem \ref{th_areagrowth_meancurvop}.

\begin{theorem}
[Slice theorem for graphs]\label{th_intro_slice_graphs}Let $M$ be a complete
manifold with boundary $\partial M\neq\emptyset$, outward pointing unit normal
$\nu$, and (at most) quadratic volume growth, i.e.,%
\[
\mathrm{vol}B_{R}^{M}\left(  o\right)  =O\left(  R^{2}\right)  ,\text{ as
}R\rightarrow+\infty,
\]
for some origin $o\in M$. Let $\Sigma$ be a graph over $M$ with non-positive
mean curvature $H\left(  x\right)  \leq0$ with respect to the orientation
given by the downward pointing Gauss map $\mathcal{N}\left(  x\right)  $.
Assume that $\partial\Sigma\cap M\times\left\{  T\right\}  =\emptyset$ for
some $T>0$ and that at least one of the following conditions is satisfied:

\begin{enumerate}
\item[(a)] $\partial\Sigma=\partial M\times\left\{  0\right\}  $ and
$\Sigma\subset M\times\lbrack0,+\infty).$

\item[(b)] $M$ and $\Sigma$ are real analytic.

\item[(c)] On $\partial\Sigma$, the Gauss map $\mathcal{N}\left(  x\right)  $
of $\Sigma$ and the Gauss map $\mathcal{N}_{0}\left(  x\right)  =\left(
-\nu\left(  x\right)  ,0\right)  $ of the boundary $\partial M\times\left\{
t\right\}  $ of any slice form an angle $\theta\left(  x\right)  \in
\lbrack-\frac{\pi}{2},\frac{\pi}{2}]$.
\end{enumerate}

\noindent If the portion of the graph $\Sigma$ contained in some half-space
$M\times\lbrack t,+\infty)$ has finite volume projection on the slice
$M\times\left\{  0\right\}  $, then $\Sigma$ is a horizontal slice of
$M\times\mathbb{R}$.
\end{theorem}

It is worth to point out that, in the setting of manifolds without boundary and for $H=0$,
half-space properties in a spirit similar to our slice-type theorems
have been obtained in the very recent paper \cite{RSS-JDG} by H. Rosenberg, F. Schulze and J. Spruck.
More precisely, they are able to show that curvature restrictions and potential theoretic properties (parabolicity) of the base manifold $M$
in the ambient product space $M\times\mathbb{R}$ force properly
immersed minimal hypersurfaces and entire minimal graphs in a half-space to be
totally geodesic slices. This holds without any further condition on their
superlevel sets.

\bigskip

The paper is organized as follows. In Section \ref{section-equilibrium} we
recall the link between parabolicity and absolute capacity of compact subsets.
We also take the occasion to give a detailed proof of the existence and
regularity of the equilibrium potentials of condensers in the setting of
manifolds with boundary. These rely on the solution of mixed boundary value
problems in non-smooth domains. Section \ref{section-ahlfors} contains the
proof of the maximum principle characterization of parabolicity and its
applications to obtain height estimates for complete CMC hypersurfaces with
boundary into Riemannian products. In Section \ref{section-stokes} we relate
the parabolicity of a manifold with boundary to the validity of the $L^{2}%
$-Stokes theorem. We also provide a weak form of this result that applies to
get slice-type results for hypersurfaces with boundary in Riemannian products.
Further slice-type results that are based on Liouville-type theorem for graphs
are also given. In the final Appendix we survey, and compare, different
notions of parabolicity for manifolds with boundary. We also exemplify how the
results of this paper can be applied in the setting of minimal surfaces. In
particular, we recover, with a deterministic proof, a result by R. Neel on the
parabolicity of minimal graphs.\bigskip

In conclusion of this introductory part we mention that there are natural and
interesting applications and extensions of the the results obtained in this
paper both to Killing graphs and to the $p$-Laplace operator.

These aspects will be presented in the forthcoming papers \cite{ILPS-Killing}
and \cite{IPS-preprint}, respectively.

\section{Capacity \& equilibrium potentials\label{section-equilibrium}}

As in the case where $M$ has no boundary, given a compact set $K$ and an open
set $\Omega$ containing $K$ the capacity of the condenser $(K, \Omega)$ is
defined by
\[
\text{cap}(K, \Omega) = \inf\{\int_{\Omega}|\nabla u|^{2} \, :\, u\in
C^{\infty}_{c}(\Omega) \,\, , u\geq1 \text{ on } K\}.
\]
When $\Omega=M$, we write $\text{cap}(K, M) = \text{cap}(K)$ and we refer to
it as the (absolute) capacity of $K$.

A simple approximation argument shows that the infimum on the right hand side
can be equivalently computed letting $u$ range over the set
\[
\{ u \in Lip_{c}(\Omega) \,:\, u=1 \text{ on }K\}
\]
or even over
\[
W_{0}(K,\Omega) = \{ u\in C(\overline{\Omega})\cap W^{1,2}_{0}(\Omega) \,:\,
u=1 \text{ on }K \},
\]
where $W^{1,2}_{0}(\Omega)=\overline{C^{\infty}_{c}(\Omega)}$. We refer to
functions in $W_{0}(K,\Omega)$ as admissible potentials for the condenser $(K,
\Omega)$.

The usual monotonicity properties of capacity hold, namely, if $K\subseteq
K_{1}$ are compact sets and $\Omega\subseteq\Omega_{1}$ are open, then
$\text{cap}(K,\Omega_{1})\leq\text{cap}(K_{1},\Omega_{1})\leq\text{cap}%
(K_{1},\Omega)$ and this allows to define first the capacity of an open set
$U\subset\Omega$ as $\text{cap} (U,\Omega) = \sup_{U\supset K, \text{ compact}
} \text{cap}(K,\Omega)$ and then the capacity of an arbitrary set
$E\subset\Omega$ as $\text{cap}(E,\Omega) = \inf_{E\subset U \text{open}%
}\text{cap}(U,\Omega)$.

We are going to show that the Liouville-type definition of parabolicity given
in the introduction is equivalent to the statement that every compact subset
has zero capacity. This depends on the construction of equilibrium potentials
for capacity, which plays a vital role also in the proof of the $L^{2}$
divergence theorem characterization of parabolicity,
Theorem~\ref{th_intro_Stokes}. It should be pointed out that while these
results are in some sense well known, we haven't been able to find a reference
which deals explicitly with matters concerning regularity up to the boundary
of these equilibrium potentials.

The following simple lemma will be useful in the proof of the proposition.

\begin{lemma}
\label{cap_exaustion} Let $D\Subset\Omega$ be open sets, and let $D_{n}$ and
$\Omega_{n}$ be a sequence of open sets such that
\[
\overline D\subseteq D_{n+1}\subseteq D_{n} \subseteq\overline{D}_{n}
\Subset\Omega_{n} \subseteq\Omega_{n+1}\subseteq\Omega, \quad\cap_{n} D_{n} =
D,\quad\cup_{n} \Omega_{n}=\Omega.
\]
Then
\begin{equation}
\label{limcap}\lim_{n} \mathrm{cap}(\overline{D}_{n}, \Omega_{n}) =
\mathrm{cap}(\overline D, \Omega).
\end{equation}

\end{lemma}

\begin{proof}
It follows from monotonicity that, for every $n$, $\mathrm{cap}(\overline
{D}_{n},\Omega_{n})$ is monotonically decreasing and greater than or equal to
$\mathrm{cap}(\overline{D}, \Omega)$ so the limit on the left hand side of
\eqref{limcap} exists and
\[
\lim_{n} \mathrm{cap}(\overline{D}_{n}, \Omega_{n}) \geq\mathrm{cap}(\overline
D, \Omega).
\]
For the converse, let $\phi\in Lip_{c}(\Omega)$ with $\phi=1$ on $\overline
D$, and for $\epsilon>0$ let
\[
\phi_{\epsilon}= \min\left\{  1, \left(  \frac{\phi-\epsilon}{1-2\epsilon
}\right)  _{+} \right\}  .
\]
By assumption, for every sufficiently large $n$ we have
\[
\overline{D}_{n} \subseteq\{x\,:\, \epsilon\leq\phi(x)\leq1-\epsilon\}
\subset\Omega_{n},
\]
and therefore $\phi_{\epsilon}$ is an admissible potential for the condenser
$(\overline D_{n}, \Omega_{n}) $ so that
\[
\int|\nabla\phi_{\epsilon}|^{2} \geq\mathrm{cap}(\overline{D}_{n}, \Omega
_{n}),
\]
whence, letting $n\to\infty$,
\[
\lim_{n} \mathrm{cap} (\overline{D}_{n}, \Omega_{n})\leq\int|\nabla
\phi_{\epsilon}|^{2} \quad\forall\epsilon>0.
\]
On the other hand, by monotone convergence,
\[
\int|\nabla\phi_{\epsilon}|^{2} = \frac{1}{(1-2\epsilon)^{2}} \int_{\{x\,:\,
\epsilon\leq\phi(x)\leq1-\epsilon\}} |\nabla\phi|^{2} \to\int_{\Omega}%
|\nabla\phi|^{2} \quad\text{as }\, \epsilon\to0,
\]
and we conclude that
\[
\lim_{n} \mathrm{cap} (\overline{D}_{n}, \Omega_{n})\leq\int|\nabla\phi|^{2},
\]
which in turn implies that
\[
\lim_{n} \mathrm{cap}(\overline{D}_{n}, \Omega_{n}) \leq\mathrm{cap}(\overline
D, \Omega).
\]

\end{proof}

\begin{proposition}
\label{equilibrium_potentials} Let $D\Subset\Omega$ be relatively compact
domains with smooth boundaries $\overline{\partial_{0}D}$ and $\overline
{\partial_{0}\Omega}$ transversal to $\partial M$. Then there exists $u\in
W_{0}(\overline{D},\Omega)\cap C^{\infty}((\Omega\setminus\overline{D}%
)\cup\partial_{1}(\Omega\setminus\overline{D}))$ such that $0\leq u\leq1$ and
\[
\mathrm{cap}(\overline{D},\Omega)=\int_{\Omega}|\nabla u|^{2}.
\]

\end{proposition}

\begin{proof}
Consider the mixed boundary value problem
\begin{equation}
\label{cap1}%
\begin{cases}
\Delta u = 0 \, \text{ in } \Omega\setminus\overline D & \\
\frac{\partial u}{\partial\nu} = 0 \, \text{ on } \partial_{1} (\Omega
\setminus\overline D)) & \\
u=0 \text{ on } \partial_{0} \Omega\,\, , u=1 \text{ on } \partial_{0} D. &
\end{cases}
\end{equation}
If follows from \cite{Lieberman- JMAA}, and the well known local regularity
theory, that \eqref{cap1} has a classical solution $u\in C(\overline{\Omega
}\setminus D) \cap C^{\infty}((\Omega\setminus\overline D) \cup\partial_{1}
(\Omega\setminus\overline D))$. By the strong maximum principle and the
boundary point lemma, it follows that $0<u<1$ on $\Omega\setminus\overline D$.
We extend $u$ to $\Omega$ by setting it equal to $1$ on $D$. To show that
$u\in W^{1,2}(\Omega)$, choose $\epsilon\in(0,1)$ such that $\epsilon$ and
$1-\epsilon$ are regular values of $u$, and let $\Omega_{\epsilon}= \{x \,:\,
u(x)\geq\epsilon\}$, $D_{\epsilon}=\{x\,:\, u(x) <1-\epsilon\}$ and
\[
u_{\epsilon}= \frac{u-\epsilon}{1-2\epsilon},
\]
so that $u_{\epsilon}\in C^{2} (\overline{\Omega}_{\epsilon}\setminus
D_{\epsilon})$ satisfies
\[%
\begin{cases}
\Delta u_{\epsilon}= 0 \, \text{ in } \Omega_{\epsilon}\setminus\overline
D_{\epsilon} & \\
\frac{\partial u_{\epsilon}}{\partial\nu} = 0 \, \text{ on } \partial_{1}
(\Omega_{\epsilon}\setminus\overline D_{\epsilon})) & \\
u_{\epsilon}=0 \text{ on } \partial_{0} \Omega_{\epsilon}\,\, , u=1 \text{ on
} \partial_{0} D_{\epsilon}, &
\end{cases}
\]

By the usual Dirichlet principle $u_{\epsilon}$ is the equilibrium potential
of the capacitor $(\overline{D}_{\epsilon}, \Omega_{\epsilon}),$ and, in
particular,
\begin{equation}
\label{cap2}\frac1{1-2\epsilon}\int_{\Omega_{\epsilon}\setminus D_{\epsilon}%
}|\nabla u|^{2} =\int_{\Omega_{\epsilon}\setminus D_{\epsilon}}|\nabla
u_{\epsilon}|^{2} = \mathrm{cap}(\overline{D}_{\epsilon}, \Omega_{\epsilon})
\end{equation}
Indeed, let $\phi\in Lip_{c}(\Omega_{\epsilon})$ with $\phi=1$ on
$D_{\epsilon}$, and let $v=u_{\epsilon}-\phi$. Then $\phi=u_{\epsilon}-v$ and
we have
\[
\int_{\Omega\epsilon} |\nabla\phi|^{2} = \int_{\Omega_{\epsilon}\setminus
D_{\epsilon}} |\nabla(u_{\epsilon}-v)|^{2}= \int_{\Omega_{\epsilon}\setminus
D_{\epsilon}}(|\nabla u_{\epsilon}|^{2} + |\nabla v|^{2} - 2\langle\nabla
u_{\epsilon},\nabla v\rangle)
\]
Since $\Delta u_{\epsilon}=0$ on ${\Omega_{\epsilon}\setminus D_{\epsilon}}$
and $v=0$ on $\partial_{0} (\Omega_{\epsilon}\setminus\overline{D}_{\epsilon
})$ while $\partial u_{\epsilon}/\partial\nu= 0$ on $\partial_{1}%
(\Omega_{\epsilon}\setminus\overline{D}_{\epsilon})$,
\[
\int_{\Omega_{\epsilon}\setminus D_{\epsilon}}\langle\nabla u_{\epsilon
},\nabla v\rangle) = - \int_{\Omega_{\epsilon}\setminus D_{\epsilon}} v\Delta
u_{\epsilon}+\int_{\partial_{0}(\Omega_{\epsilon}\setminus D_{\epsilon}%
)\cup\partial_{1}(\Omega_{\epsilon}\setminus D_{\epsilon})} \langle\nabla
u_{\epsilon},\nu\rangle v =0,
\]
so that
\[
\int_{\Omega\epsilon} |\nabla\phi|^{2} = \int_{\Omega_{\epsilon}\setminus
D_{\epsilon}}(|\nabla u_{\epsilon}|^{2} + |\nabla v|^{2})\geq\int
_{\Omega_{\epsilon}\setminus D_{\epsilon}} |\nabla u_{\epsilon}|^{2},.
\]
as claimed.

Letting $\epsilon\rightarrow0$ $\Omega_{\epsilon}\setminus D_{\epsilon
}\nearrow\Omega\setminus D$, so that, by monotone convergence, the integral in
\eqref{cap2} converges to
\[
\int_{\Omega\setminus D}|\nabla u|^{2}.
\]
On the other hand, by the previous lemma,
\[
\mathrm{cap}(\overline{D}_{\epsilon},\Omega_{\epsilon})\rightarrow
\mathrm{cap}(\overline{D},\Omega),\quad\text{as }\,\epsilon\rightarrow0
\]
and we conclude that $u\in W^{1,2}(\Omega)$ so that, in fact, $u\in
W_{0}(\overline{D},\Omega)$ and
\[
\int_{\Omega}|\nabla u|^{2}=\mathrm{cap}(\overline{D},\Omega),
\]
as required to complete the proof.
\end{proof}

\begin{remark}
It is worth to point out that the equilibrium potential $u$ of the capacitor
$(\overline{D},\Omega)$ constructed using Liebermann approach coincides with
the one obtained by applying the direct calculus of variations to the energy
functional on the closed convex space
\[
W_{\Gamma}^{1,2}\left(  \Omega\backslash\overline{D}\right)  =\left\{  u\in
W^{1,2}\left(  \Omega\right)  :\left.  u\right\vert _{\partial_{0}D}=0\text{
and }\left.  u\right\vert _{\partial_{0}\Omega}=1\right\}  .
\]
Here, Dirichlet data are understood in the trace sense. Thanks to the global
$W^{1,2}$-regularity established in Proposition \ref{equilibrium_potentials},
this follows e.g. either from maximum principle considerations or from the
convexity of the energy functional.
\end{remark}

\begin{proposition}
\label{equilibrium_potentials2} Let $D$ be a relatively compact domain and let
$\Omega_{j}$ be an increasing exhaustion of $M$ by relatively compact open
domains with $\overline{D}\subset\Omega_{1}$. Assume that $\overline
{\partial_{0}D}$ and $\overline{\partial_{0}\Omega_{j}}$ are smooth and
transversal to $\partial M$, and for every $j$, let $u_{j}$ be the equilibrium
potential of the capacitor $(\overline{D},\Omega_{j})$ constructed in
Proposition~\ref{equilibrium_potentials}. Then $u_{j}$ converges monotonically
to a function $u\in C(M)\cap W_{loc}^{1,2}\cap C^{2}(M\setminus\overline{D})$
such that $0\leq u\leq1$, $u=1$ on $\overline{D}$, $u$ is harmonic on
$M\setminus\overline{D}$, $\partial u/\partial\nu=0$ on $\partial
_{1}(M\setminus\overline{D})$ and $u$ is a weak Neumann supersolution of the
Laplace equation on $M$. Moreover $\nabla u\in L^{2}(M)$,
\[
\mathrm{cap}(\overline{D})=\int_{M}|\nabla u|^{2}.
\]

\end{proposition}

\begin{proof}
Extend $u_{j}$ to all of $M$ by setting it equal to zero in $M\setminus
\Omega_{j}$. It follows by the comparison principle that $0\leq u_{j}\leq
u_{j+1}\leq1$ in $\Omega_{j}\setminus\overline{D}$, and therefore the sequence
$u_{j}$ converges monotonically to a function $u$. Note that since
$u_{j}(x)\leq u(x)\leq1$ and $u_{j}(x)\rightarrow1$ as $x\rightarrow
y\in\partial_{0}D$ is follows that $u$ is continuous on $\overline{D}$ and
there it is equal to $1$. Moreover, by the Schauder type estimate contained in
Lemma 1 in \cite{Lieberman- JMAA}, for every $\alpha\in(0,1)$, every $j_{o}$
and every sufficiently small $\eta>0$ there exists a constant $C$ depending
only on $\alpha$ $\eta$, $j_{o}$ and on the geometry of $M$ in a neighborhood
of $B_{j_{o},\eta}=\{x\in\Omega_{j_{o}}\setminus\overline{D}\,:\,\mathrm{dist}%
(x,\partial_{0}D\cup\partial_{0}\Omega_{jo})\geq\eta\}
$ such that, for every $j\geq j_{o}$
\[
||u_{j}||_{C^{2,\alpha}(B_{\eta})}\leq C\sup_{B_{\eta/2}}|u_{j}(x)|.
\]
It follows immediately that (possibly passing to a subsequence) the sequence
$u_{j}$ converges in $C^{2}(B_{j_{o},\eta})$ for every $j_{o}$ and $\eta>0$ so
that the limit function $u$ is harmonic in $\mathrm{int}\,M\setminus
\overline{D}$ and $C^{2}$ up to $\partial_{1}(M\setminus\overline{D})$ where
it satisfies the Neumann boundary condition $\partial u/\partial\nu=0$.
Summing up, $u\in C^{0}(M\setminus D)\cup C^{2}((M\setminus\overline{D}%
)\cup(\partial_{1}(M\setminus\overline{D})))$ is a classical solution of the
mixed boundary problem
\[%
\begin{cases}
\Delta u\geq0\text{ on }M\setminus\overline{D}\\
\frac{\partial u}{\partial\nu}\leq0\,\text{ on }\,\partial_{1}(M\setminus
\overline{D})\\
u=1\text{ on }\partial_{0}D\\
0\leq u\leq1.
\end{cases}
\]
On the other hand, since
\[
\int_{\Omega_{j}}|\nabla u_{j}|^{2}=\mathrm{cap}(\overline{D},\Omega
_{j})\searrow\mathrm{cap}(\overline{D}),
\]
the sequence $u_{j}\in C^{0}(M)\cap W_{c}^{1,2}(M)$ converges pointwise to $u$
and $\nabla u_{j}$ is bounded in $L^{2}(M).$ It follows easily (see, e.g.,
Lemma 1.33 in \cite{HKM}) that $\nabla u\in L^{2}(M)$ and $\nabla
u_{j}\rightarrow\nabla u$ weakly in $L^{2}$. By the weak lower semicontinuity
of the energy functional, it follows that
\[
\int_{M}|\nabla u|^{2}\leq\liminf_{j}\int_{M}|\nabla u_{j}|^{2}=\mathrm{cap}%
(\overline{D})
\]
On the other hand, By Mazur's Lemma, a convex combination $\tilde{u}_{j}$ of
the $u_{j}$ is such that $\nabla\tilde{u}_{j}\rightarrow\nabla u$ strongly in
$L^{2}(M),$ and since each $\tilde{u}_{j}\in C^{0}\cap W^{1,2}(M)$ is
compactly supported, and equal to $1$ on $\overline{D}$, it admissible for the
capacitor $(\overline{D},M)$ and we deduce that
\[
\int_{M}|\nabla u|^{2}=\lim\int_{M}|\nabla\tilde{u}_{j}|^{2}\geq
\mathrm{cap}(\overline{D}),
\]
and we conclude that
\[
\int_{M}|\nabla u|^{2}=\mathrm{cap}(\overline{D}),
\]
as required.

Finally, assume that $u$ is non-constant so that, by the strong maximum
principle, $u<1$ in $M\setminus\overline D$. Let $\eta_{n}\to1$ be a sequence
of regular values of $u$, and set $\Gamma_{n}=\{x \,:\, u(x)< \eta_{n}\}.$
Using the fact that $\Delta u=0$ on $\Gamma_{n} \subset M\setminus\overline
D$, $\partial u/\partial\nu=0$ on $\partial_{1} \Gamma_{n}$ and $\partial
u/\partial\nu\geq0$ on $\partial_{0}\Gamma_{n}$, given $0\leq\rho\in
C^{\infty}_{c} (M)$, we compute
\[
\int_{M} \langle\nabla u,\nabla\rho\rangle= \lim_{n} \int_{\Gamma_{n}}
\langle\nabla u, \nabla\rho\rangle= \lim_{n} \{- \int_{\Gamma_{n}} \rho\Delta
u + \int_{\partial_{0} \Gamma_{n} \cup\partial_{1}\Gamma_{n}}
\!\!\!\!\!\!\!\!\rho\langle\nabla u, \nu\rangle\} \geq0,
\]
and $u$ is a weak Neumann supersolution of the Laplace equation on $M$.
\end{proof}

We then obtain the announced equivalent characterization of parabolicity.

\begin{theorem}
\label{capacity_parabolicity} Let $(M,\langle\,,\rangle)$ be a connected
Riemannian manifold with (possibily empty) boundary $\partial M$. The
following are equivalent:

\begin{itemize}
\item[(i)] The capacity of every compact set $K$ in $M$ is zero.

\item[(ii)] For every relatively compact open domain $D\Subset M$ there exists
an increasing sequence of functions $h_{j}\in C^{0}(M)\cap W^{1,2}_{c}(M)$
with $h_{j}=1$ on $D$, $0\leq h_{j}\leq h_{j+1}\leq1$, $h_{j}$ harmonic in the
set $\{x: 0<h_{j}(x)<1\}\cap\mathrm{int} M$, such that
\[
\int_{M} |\nabla h_{j}|^{2}\to0 \quad\text{ as }\, j\to+\infty.
\]

\item[(iii)] $M$ is parabolic.
\end{itemize}
\end{theorem}

\begin{proof}
(i) $\Rightarrow$ (ii). Assume first that $\mathrm{cap}(K)=0$ for every
compact set $K$ in $M$, let $D$ be as in (ii) and let $\Omega_{j}$ be an
increasing exhastion of $M$ by relatively compact open set with smooth
boundary transversal to $\partial M$ with $\overline D\subset\Omega_{1}$. For
every $j$ let $u_{j}$ be the equilibrium potential of the capacitor
$(\overline D, \Omega_{j})$, and extend $u_{j}$ to be $0$ off $\Omega_{j}$.
Then $u_{j}$ has the regularity properties listed in (ii), and, by
Proposition~\ref{equilibrium_potentials},
\[
\int|\nabla u_{j}|^{2} = \mathrm{cap}(\overline D, \Omega_{j}) \to
\mathrm{cap}(\overline D)=0.
\]

(ii) $\Rightarrow$ (i) Conversely, assume that (ii) holds. Clearly it suffices
to prove that $\mathrm{cap}(\overline D)=0 $ for every relatively compact open
domain $D$ with smooth boundary transversal to $\partial M$. Choose an
increasing exhaustion of $M$ by relatively compact domains $\Omega_{j}$ with
smooth boundary transversal to $\partial M$ such that $\mathrm{supp} u_{j}
\Subset\Omega_{j}$. Then
\[
\mathrm{cap} (\overline D) = \lim_{j} \mathrm{cap} (\overline D, \Omega_{j})
\leq\lim_{j} \int_{\Omega_{j}} |\nabla u_{j}|^{2} \to0,
\]
as required.

(i) $\Rightarrow$ (iii) Suppose that $\mathrm{cap}(K)= 0 $ for every compact
set in $M$, and let $u\in C^{0}(M)\cap W^{1,2}_{loc}(M)$ satisfy, in the weak
Neumann sense,
\begin{equation}
\label{weak_liou}%
\begin{cases}
\Delta u\geq0\\
\frac{\partial u}{\partial\nu}\leq0 \, \text{ on } \, \partial M\\
\sup_{M} u<+\infty.
\end{cases}
\end{equation}
Let $v=\sup_{M} u- u +1$, so that $v\geq1$ and, by definition of weak solution
of the differential problem \eqref{weak_liou}, $v$ satisfies
\[
\int\langle\nabla v, \nabla\rho\rangle\geq0 \quad\forall0\leq\rho\in
C^{0}(M)\cap W^{1,2}_{0}(M).
\]
Next, for every relatively compact domain $D$, let $\varphi\in Lip_{c}(M)$
with $\varphi=1$ on $D$, and $0\leq\varphi\leq1$. Using $\rho= \varphi^{2}
v^{-1}\in C^{0}(M) \cap W^{1,2}_{c}(M)$ as a test function we have
\[%
\begin{split}
0\leq\int\langle v, \nabla\rho\rangle &  = 2\int\varphi\langle v^{-1} \nabla
v, \nabla\varphi\rangle- \int\varphi^{2} |v^{-1}\nabla v |^{2}\\
&  \leq2\int\varphi|v^{-1}\nabla v| |\nabla\varphi| - \int\varphi^{2}
|v^{-1}\nabla v |^{2}.
\end{split}
\]
Rearranging, using Young's inequality $2ab \leq2a^{2} + \frac12 b^{2}$, and
recalling that $\varphi=1$ on $D$ we obtain
\[
\int_{D} |v^{-1}\nabla v |^{2}\leq4\int|\nabla\varphi|^{2},
\]
and taking the inf of the right hand side over all $Lip_{c}$ function
$\varphi$ which are equal to $1$ on $D$ we conclude that
\[
\int_{D} |v^{-1}\nabla v |^{2}\leq4\mathrm{cap}(\overline{D}) = 0
\]
Thus $v$ and therefore $u$ is constant on every relatively compact domain $D$.
Thus $u$ is constant on $M$, and $M$ is parabolic in the sense of
Definition~\ref{def_parab}.

(iii) $\Rightarrow$ (i) Assume by contradiction that there exists compact set
$K$ with nonzero capacity. Without loss of generality we can suppose that $K$
is the closure of a relatively compact open domain $D$ with smooth boundary
$\partial_{0}D$ transversal to $\partial M$. Let $u$ be the equilibrium
potential of $\overline D$ constructed in
Proposition~\ref{equilibrium_potentials2}, which is non-constant since the
capacity of $\overline D$ is positive. But then $u\in C^{0}(M)\cap W^{1,2}(M)$
is a non-constant bounded weak Neumann superharmonic function, contradicting
the assumed parabolicity of $M$.
\end{proof}

\section{Maximum principles \& height estimates\label{section-ahlfors}}

It is a classical result by L.V. Ahlfors that a Riemannian manifold $N$
(without boundary) is parabolic if and only if, for every domain $D\subseteq
N$ with $\partial D\neq\emptyset$ and for every bounded above, subharmonic
function $u$ on $D$ it holds that $\sup_{D}u=\sup_{\partial D}u$. The result
has been extended in the setting of $p$-parabolicity in \cite{PST}. This
section aims to provide a new form of the Ahlfors characterization which is
valid on manifolds with boundary. This, in turn, will be used to obtain
estimate of the height function of complete hypersurfaces with constant mean
curvature (CMC for short)\ immersed into product spaces of the form
$N\times\mathbb{R}$.

\subsection{Global maximum principles\label{subsection-ahlfors}}

We are going to prove the Ahlfors-type characterization of parabolicity stated
in Theorem \ref{th_intro_Ahlfors}. Actually, a version of this global maximum
principle involving the whole manifold and without any Neumann condition will
be crucial in the geometric applications. This is the content of Theorem
\ref{th_ahlfors-wholeM} that will be proved at the end of the section.

\begin{proof}
[Proof (of Theorem \ref{th_intro_Ahlfors})]Assume first that $M$ is parabolic
and suppose, by contradiction, that there exists a domain $D\subseteq M$ and a
function $u$ as in the statement of the Theorem, such that%
\[
\sup_{D}u>\sup_{\partial_{0}D}u.
\]
Let $\varepsilon>0$ be so small that%
\[
\sup_{D}u>\sup_{\partial_{0}D}u+\varepsilon.
\]
Then, the open set $D_{\varepsilon}=\left\{  x\in D:u>\sup_{D}u-\varepsilon
\right\}  \neq\emptyset$ satisfies $\overline{D}_{\varepsilon}\subset D$ and,
therefore,%
\[
u_{\varepsilon}=%
\begin{cases}
\max\left\{  u,\sup_{D}u-\varepsilon\right\}  & \text{on }D\\
\sup_{D}u-\varepsilon & \text{on }M\backslash D
\end{cases}
\]
well defines a $C^{0}\left(  M\right)  \cap W_{loc}^{1,2}\left(  M\right)
$-subsolution of the Laplace equation on $M$. Furthermore, $\sup
_{M}u_{\varepsilon}=\sup_{D}u<+\infty$. It follows from the very definition of
parabolicity that $u_{\varepsilon}$ is constant on $M$. In particular, if we
suppose to have chosen $\varepsilon>0$ in such a way that $\sup_{D}%
u-\varepsilon$ is not a local maximum for $u$, then $u_{\varepsilon}=\sup
_{D}u-\varepsilon$ on $\partial D_{\varepsilon}\neq\emptyset$ and we conclude%
\[
u\equiv\sup_{D}u-\varepsilon\text{, on }D,
\]
which is absurd.

Suppose now that, for every domain $D\subseteq M$ with $\partial_{0}%
D\neq\emptyset$ and for every $u\in C^{0}\left(  \overline{D}\right)  \cap
W_{loc}^{1,2}\left(  D\right)  $ satisfying, in the weak Neumann sense,%
\[
\left\{
\begin{array}
[c]{ll}%
\Delta u\geq0 & \text{on }D\\
\dfrac{\partial u}{\partial\nu}\leq0 & \text{on }\partial_{1}D\\
\sup_{D}u<+\infty, &
\end{array}
\right.
\]
it holds%
\[
\sup_{D}u=\sup_{\partial_{0}D}u.
\]
By contradiction assume that $M$ is not parabolic. Then, there exists a
non-constant function $v\in C^{0}\left(  M\right)  \cap W_{loc}^{1,2}\left(
M\right)  $ satisfying
\[
\left\{
\begin{array}
[c]{ll}%
\Delta v\geq0 & \text{on }M\\
\dfrac{\partial v}{\partial\nu}\leq0 & \text{on }\partial M\\
v^{\ast}=\sup_{M}v<+\infty. &
\end{array}
\right.
\]
Given $\eta<v^{\ast}$ consider the domain $\Omega_{\eta}=\{x\in M:v(x)>\eta
\}\neq\emptyset$. We can choose $\eta$ sufficiently close to $v^{\ast}$ in
such a way that $\mathrm{int}M\not \subseteq \Omega_{\eta}$. In particular,
$\partial\Omega_{\eta}\subseteq\left\{  v=\eta\right\}  $ and $\partial
_{0}\Omega_{\eta}\neq\emptyset$. \ Now, $v\in C^{0}\left(  \overline{\Omega
}_{\eta}\right)  \cap W_{loc}^{1,2}\left(  \Omega_{\eta}\right)  $ is a
bounded above weak Neumann subsolution on $\partial_{1}\Omega_{\eta}$.
Moreover,%
\[
\sup_{\partial_{0}\Omega_{\eta}}{v}=\eta<\sup_{\Omega_{\eta}}{v},
\]
contradicting our assumptions.
\end{proof}

\begin{remark}
If we take $D=M$ in the first half of the above proof then we immediately
realize that the Neumann boundary condition plays no role. This suggests the
validity of the following restricted form of the maximum principle that was
adopted by F.R. De Lima \cite{DeLi} as a definition of a weak notion of
parabolicity; see Appendix \ref{appendix-different}.
\end{remark}

\begin{proof}
[Proof (of Theorem \ref{th_ahlfors-wholeM})]If, by contradiction,%
\[
\sup_{M}u>\sup_{\partial M}u
\]
then, we can choose $\varepsilon>0$ so small that%
\[
\sup_{M}u>\sup_{\partial M}u-2\varepsilon.
\]
Define $u_{\varepsilon}\in C^{0}\left(  M\right)  \cap W_{loc}^{1,2}\left(
M\right)  $ by setting%
\[
u_{\varepsilon}=\left\{
\begin{array}
[c]{ll}%
\max\left(  u,\sup_{M}u-\varepsilon\right)  & \text{on }\Omega_{2\varepsilon
}\\
\sup_{M}u-\varepsilon & \text{on }M\backslash\Omega_{2\varepsilon},
\end{array}
\right.
\]
where we have set%
\[
\Omega_{2\varepsilon}=\left\{  x\in M:u\left(  x\right)  >\sup_{M}%
u-2\varepsilon\right\}  .
\]
Since $\overline{\Omega}_{2\varepsilon}\subset\mathrm{int}M$, we have that
$u_{\varepsilon}$ is constant in a neighborhood of $\partial M$. Since $\Delta
u \geq0 $ weakly on $\mathrm{int} M$, it follows that $u_{\epsilon}$ is a weak
Neumann subsolution on $M$. Moreover, $\sup_{M}u_{\varepsilon}=\sup
_{M}u<+\infty$ so that, by parabolicity, $u_{\varepsilon}\equiv\sup
_{M}u-\varepsilon$, a contradiction.
\end{proof}

\subsection{Height estimates for CMC hypersurfaces in product
spaces\label{section-heightestimates}}

We now present some applications of this global maximum principle to get
height estimates both for $H$-hypersurfaces with boundary in product spaces
and for $H$-graphs over manifolds with boundary. By an $H$-hypersurfaces of
$N\times\mathbb{R}$ we mean and oriented hypersurface $\Sigma$ with constant
mean curvature $H$ with respect to a choice of its Gauss map. An $H$-graph
over the $m$-dimensional Riemannian manifold $M$ with boundary $\partial
M\neq\emptyset$ is an embedded $H$-hypersurfaces given by $\Sigma=\Gamma
_{u}\left(  M\right)  $ where $\Gamma_{u}:M\rightarrow M\times\mathbb{R}$ \ is
defined, as usual, by $\Gamma_{u}\left(  x\right)  =\left(  x,u\left(
x\right)  \right)  $, for some smooth function $u:M\rightarrow\mathbb{R}$. The
downward (pointing) unit normal to $\Sigma$ is defined by
\[
\mathcal{N}=\frac{1}{\sqrt{1+\left\vert \nabla_{M}u\right\vert ^{2}}}\left(
\nabla_{M}u,-1\right)  .
\]
With respect to $\mathcal{N}$, the mean curvature of the graph writes as%
\[
H=-\frac{1}{m}\operatorname{div}_{M}\left(  \frac{\nabla_{M}u}{\sqrt
{1+\left\vert \nabla_{M}u\right\vert ^{2}}}\right)  .
\]
On the other hand, let \thinspace$M_{\Sigma}$ be the original manifold $M$
endowed with the metric pulled back from $M\times\mathbb{R}$ via $\Gamma_{u}$.
Then, it is well known that the mean curvature vector field of the isometric
immersion $\Gamma_{u}$%
\[
\mathbf{H}\left(  x\right)  =H\left(  x\right)  \mathcal{N}\left(  x\right)
\]
satisfies%
\[
\Delta_{\Sigma}\Gamma_{u}=m\mathbf{H},
\]
where $\Delta_{\Sigma}$ \ denotes the Laplacian on manifold-valued maps. Since
$\Delta_{\Sigma}$ is linear with respect to the Riemannian product structure
in the codomain, from the above we also get%
\begin{align}
\Delta_{\Sigma}u  &  =\frac{1}{\sqrt{1+\left\vert \nabla_{M}u\right\vert ^{2}%
}}\operatorname{div}_{M}\left(  \frac{\nabla_{M}u}{\sqrt{1+\left\vert
\nabla_{M}u\right\vert ^{2}}}\right) \label{lap-meancurv}\\
&  =-\frac{m}{\sqrt{1+\left\vert \nabla_{M}u\right\vert ^{2}}}H\left(
x\right) \nonumber
\end{align}

With this preparation, we begin by noting the following version of Lemma 1 in
\cite{LW}.

\begin{lemma}
\label{prop_height} Let $N$ be an $m$-dimensional complete manifold without
boundary and let $M\subset N$ be a closed domain with smooth boundary
$\partial M\neq\emptyset$. Consider a graph $\Sigma=\Gamma_{u}\left(
M\right)  \subset N\times\mathbb{R}$ over $M$ with smooth boundary%
\[
\partial\Sigma\subset M\times\left\{  0\right\}  .
\]
Assume that%
\[
\sup_{M}|u|+\sup_{M}|H|<+\infty.
\]
Then there exists a constant $C=C(m,\sup_{M}|u|,\sup_{M}|H|)>0$ such that, for
every $\delta>0$ and $R>1$,%
\[
\mathrm{vol}B_{R}^{\Sigma}\left(  \bar{p}\right)  \leq C\left(  1+\frac
{1}{\delta R}\right)  \mathrm{vol}\left(  M\cap B_{\left(  1+\delta\right)
R}^{N}\left(  \bar{x}\right)  \right)  ,
\]
where $\bar{x}$ is a reference point in $N$ and $\bar{p}=(\bar{x},u(\bar{x}%
))$. Moreover, the following estimate%
\[
\mathrm{vol}B_{R}^{\Sigma}\left(  \bar{p}\right)  \leq C\left\{
\mathrm{vol}B_{R}^{N}\left(  \bar{x}\right)  +\mathrm{Area}\left(  \partial
B_{R}^{N}\left(  \bar{x}\right)  \right)  \right\}
\]
holds for almost every $R>1$.
\end{lemma}

\begin{proof}
Note that%
\begin{align*}
d_{\Sigma}\left(  \left(  \bar{x},u\left(  \bar{x}\right)  \right)  ,\left(
x,u\left(  x\right)  \right)  \right)   &  \geq d_{N\times\mathbb{R}}\left(
\left(  \bar{x},u\left(  \bar{x}\right)  \right)  ,\left(  x,u\left(
x\right)  \right)  \right)  \\
&  \geq\max\left\{  d_{N}\left(  \bar{x},x\right)  +\left\vert u\left(
\bar{x}\right)  -u\left(  x\right)  \right\vert \right\}  .
\end{align*}
Set $\bar{p}=(\bar{x},u(\bar{x}))$. Therefore%
\begin{align*}
B_{R}^{\Sigma}\left(  \bar{p}\right)   &  \subseteq\Sigma\cap B_{R}%
^{N\times\mathbb{R}}\left(  \bar{p}\right)  \\
&  \subseteq(M\cap B_{R}^{N}\left(  \bar{x}\right)  )\times\left(  -R+u\left(
\bar{x}\right)  ,R+u\left(  \bar{x}\right)  \right)
\end{align*}
and it follows that%
\begin{align}
\mathrm{vol}B_{R}^{\Sigma}\left(  \bar{p}\right)   &  =\int_{\Pi_{N}\left(
B_{R}^{\Sigma}\left(  \bar{p}\right)  \right)  }\sqrt{1+\left\vert \nabla
u\right\vert ^{2}}d\mathrm{vol}_{N}\label{LW-volest1}\\
&  \leq\int_{M\cap B_{R}^{N}\left(  \bar{x}\right)  }\sqrt{1+\left\vert \nabla
u\right\vert ^{2}}d\mathrm{vol}_{N}\nonumber\\
&  =\int_{M\cap B_{R}^{N}\left(  \bar{x}\right)  }\frac{\left\vert \nabla
u\right\vert ^{2}}{\sqrt{1+\left\vert \nabla u\right\vert ^{2}}}%
d\mathrm{vol}_{N}+\int_{M\cap B_{R}^{N}\left(  \bar{x}\right)  }\frac{1}%
{\sqrt{1+\left\vert \nabla u\right\vert ^{2}}}d\mathrm{vol}_{N}\nonumber\\
&  \leq\int_{M\cap B_{R}^{N}\left(  \bar{x}\right)  }\frac{\left\vert \nabla
u\right\vert ^{2}}{\sqrt{1+\left\vert \nabla u\right\vert ^{2}}}%
d\mathrm{vol}_{N}+\mathrm{vol}(M\cap B_{R}^{N}\left(  \bar{x}\right)
).\nonumber
\end{align}
Here $\Pi_{N}:\Sigma\rightarrow N$ denotes the projection on the $N$ factor.
Now, for any $\delta>0$, we choose a cut-off function $\rho$ as follows:%
\[
\rho(x)=%
\begin{cases}
1 & \mathrm{on}\quad B_{R}(\bar{x})\\
\frac{\left(  1+\delta\right)  R-r(x)}{\delta R} & \mathrm{on}\quad B_{\left(
1+\delta\right)  R}(\bar{x})\backslash B_{R}(\bar{x})\\
0 & \mathrm{elsewhere},
\end{cases}
\]
where $r(x)$ denotes the distance function on $N$ from a reference point
$\bar{x}$. Since
\[
X=\rho u\frac{\nabla u}{\sqrt{1+\left\vert \nabla u\right\vert ^{2}}}%
\]
is a compactly supported vector field that vanishes on $\partial M$ and on
$\partial B_{\left(  1+\delta\right)  R}^{N}\left(  \bar{x}\right)  $, as an
application of the divergence theorem we get
\begin{align*}
0= &  \int_{M\cap B_{\left(  1+\delta\right)  R}^{N}\left(  \bar{x}\right)
}\mathrm{div}(X)d\mathrm{vol}_{N}\\
= &  -m\int_{M\cap B_{\left(  1+\delta\right)  R}^{N}\left(  \bar{x}\right)
}\rho Hud\mathrm{vol}_{N}+\int_{M\cap B_{\left(  1+\delta\right)  R}%
^{N}\left(  \bar{x}\right)  }\frac{\rho\left\vert \nabla u\right\vert ^{2}%
}{\sqrt{1+\left\vert \nabla u\right\vert ^{2}}}d\mathrm{vol}_{N}\\
&  -\frac{1}{\delta R}\int_{M\cap(B_{\left(  1+\delta\right)  R}^{N}\left(
\bar{x}\right)  \backslash B_{R}^{N}\left(  \bar{x}\right)  )}u\frac
{\langle\nabla u,\nabla r\rangle}{\sqrt{1+\left\vert \nabla u\right\vert ^{2}%
}}d\mathrm{vol}_{N}.
\end{align*}
Hence
\begin{align*}
\int_{M\cap B_{R}^{N}\left(  \bar{x}\right)  }\frac{\left\vert \nabla
u\right\vert ^{2}}{\sqrt{1+\left\vert \nabla u\right\vert ^{2}}}%
d\mathrm{vol}_{N}\leq &  \int_{M\cap B_{\left(  1+\delta\right)  R}^{N}\left(
\bar{x}\right)  }\frac{\rho\left\vert \nabla u\right\vert ^{2}}{\sqrt
{1+\left\vert \nabla u\right\vert ^{2}}}d\mathrm{vol}_{N}\\
\leq &  m\sup_{M}|u|\sup_{M}|H|\mathrm{vol}(M\cap B_{\left(  1+\delta\right)
R}^{N}\left(  \bar{x}\right)  )\\
+ &  \frac{\sup_{M}|u|}{\delta R}\mathrm{vol}(M\cap(B_{\left(  1+\delta
\right)  R}^{N}\left(  \bar{x}\right)  \backslash B_{R}^{N}\left(  \bar
{x}\right)  )).
\end{align*}
Inserting this latter into (\ref{LW-volest1}) gives, for every $R>1$,
\begin{align*}
\mathrm{vol}B_{R}^{\Sigma}\left(  \bar{p}\right)  \leq &  C\left\{
\mathrm{vol}(M\cap B_{R}^{N}\left(  \bar{x}\right)  )+\mathrm{vol}(M\cap
B_{\left(  1+\delta\right)  R}^{N}\left(  \bar{x}\right)  )\right.  \\
&  \left.  +\frac{1}{\delta R}\mathrm{vol}(M\cap(B_{\left(  1+\delta\right)
R}^{N}\left(  \bar{x}\right)  \backslash B_{R}^{N}\left(  \bar{x}\right)
))\right\}  .
\end{align*}
To conclude, we let $\delta\rightarrow0$ and we use the co-area formula.
\end{proof}

\begin{remark}
We note that, actually, the somewhat weaker conclusions%
\[
\mathrm{vol}B_{R}^{\Sigma}\left(  \bar{p}\right)  \leq C\left(  1+\frac
{1}{\delta}\right)  \mathrm{vol}\left(  M\cap B_{\left(  1+\delta\right)
R}^{N}\left(  \bar{x}\right)  \right)  ,
\]
and%
\[
\mathrm{vol}B_{R}^{\Sigma}\left(  \bar{p}\right)  \leq C\left\{
\mathrm{vol}B_{R}^{N}\left(  \bar{x}\right)  +R\mathrm{Area}\left(  \partial
B_{R}^{N}\left(  \bar{x}\right)  \right)  \right\}
\]
hold under the assumption
\[
\sup_{M}|uH|<+\infty.
\]
Indeed, to overcome the problem that $u$ can be unbounded, following the proof
in the minimal case $H\equiv0$, one can apply the divergence theorem to the
vector field
\[
X=\rho u_{\sqrt{2}R}\frac{\nabla u}{\sqrt{1+\left\vert \nabla u\right\vert
^{2}}},
\]
where $u_{R}$ is defined as
\[
u_{R}=%
\begin{cases}
-R & \text{if}\quad u(x)<-R\\
u(x) & \text{if}\quad|u(x)|<R\\
R & \text{if}\quad u(x)>R.
\end{cases}
\]

\end{remark}

\begin{remark}
It could be interesting to observe that, in certain situations, an improved
version of Lemma \ref{prop_height} can be obtained from the a-priori gradient
estimates due to N. Koreevar, X.-J. Wang and J. Spruck, \cite{Ko, Wa, Sp}. See
also \cite{RSS-JDG} where the injectivity radius assumption has been removed.
More precisely, we have the next simple result. We explicitly note that, with
respect to Lemma \ref{prop_height}, no assumption on $\partial\Sigma$ is
required. Moreover, the volume estimate involves the same radius $R>0$ without
any further contribution.
\end{remark}

\begin{lemma}
\label{lemma_volume}Let $\left(  N,g\right)  $ be a complete, $m$-dimensional
Riemannian manifold (without boundary) satisfying $Sec_{N}\geq-K$ and let
$M\subset N$ be a closed domain with smooth boundary $\partial M\neq\emptyset
$. Suppose we are given a vertically bounded graph $\Sigma_{\varepsilon
}=\Gamma_{u}\left(  \mathcal{U}_{\varepsilon}\left(  M\right)  \right)  $ with
bounded mean curvature $H$, parametrized over an $\varepsilon$-neighborhood
$\mathcal{U}_{\varepsilon}\left(  M\right)  $ of $M$. Let $\Sigma=\Gamma
_{u}\left(  M\right)  .$ Then, there exists a constant $C=C\left(m,
\varepsilon,H,K,\sup_{M}\left\vert u\right\vert ,\sup_{M}\left\vert
H\right\vert \right)  >0$ such that%
\[
\mathrm{vol}B_{R}^{\Sigma}\left(  \bar{p}\right)  \leq C\mathrm{vol}\left(
M\cap B_{R}^{N}\left(  \bar{x}\right)  \right)  ,
\]
for every $R>0$, where $\bar{x}\in\mathrm{int}M$ is a reference point and
$\bar{p}=\left(  \bar{x},u\left(  \bar{x}\right)  \right)  $.
\end{lemma}

\begin{proof}
Indeed, since%
\[
\mathrm{vol}B_{R}^{\Sigma}\left(  \bar{p}\right)  =\int_{\Pi_{N}\left(
B_{R}^{\Sigma}\left(  \bar{p}\right)  \right)  }\sqrt{1+\left\vert \nabla
u\right\vert ^{2}}d\mathrm{vol}_{N}\leq\int_{M\cap B_{R}^{N}\left(  \bar
{x}\right)  }\sqrt{1+\left\vert \nabla u\right\vert ^{2}}d\mathrm{vol}_{N},
\]
we have only to show that $\left\vert \nabla u\right\vert $ is uniformly
bounded on $M$. To this end, note that $u:\mathcal{U}_{\varepsilon}\left(
M\right)  \rightarrow\mathbb{R}$ is a bounded function defining a bounded mean
curvature graph $\Gamma_{u}\left(  \mathcal{U}_{\varepsilon
}\left(  M\right)  \right)  $. Therefore, we can apply Theorem 1.1 in
\cite{Sp} to either $w\left(  x\right)  =\sup_{M}u-u\left(  x\right)  \geq0$
or $w\left(  x\right)  =u\left(  x\right)  -\inf_{M}u\geq0$ and obtain that,
in fact, $\left\vert \nabla^{M}u\right\vert $ is uniformly bounded on every
ball $B_{\varepsilon/2}^{N}\left(  x\right)  \subset\mathcal{U}_{\varepsilon
}\left(  M\right)  $, with $x\in M$. This completes the proof.
\end{proof}

\bigskip

Lemma \ref{prop_height} allows to prove Theorem \ref{th_intro_hest_graph}
stated in the Introduction.

\begin{proof}
[Proof (of Theorem \ref{th_intro_hest_graph})]Observe first that, according to
Lemma \ref{prop_height}, since $N$ has quadratic volume growth, so has
$\Sigma$. In particular, by Theorem \ref{th_growth}, if we denote by
$M_{\Sigma}$ the original domain $M$ endowed with the metric pulled back from
$\Sigma$ via $\Gamma_{u}$, we conclude that $M_{\Sigma}$ is parabolic.
Consider now the real-valued function $w\in C^{0}\left(  M_{\Sigma}\right)
\cap C^{\infty}\left(  \mathrm{int}M_{\Sigma}\right)  $ defined by%
\[
w\left(  x\right)  =Hu\left(  x\right)  -\frac{1}{\sqrt{1+\left\vert \nabla
u\left(  x\right)  \right\vert ^{2}}}.
\]
Since $Ric_{N}\geq0$, it is well known that $w$ is subharmonic; see e.g.
\cite{AD}. Moreover, $w\leq0$ on $\partial M_{\Sigma}$ and $\sup_{M_{\Sigma}%
}w\leq H\sup_{M}u<+\infty$. It follows from Theorem \ref{th_ahlfors-wholeM}
that%
\[
\sup_{M_{\Sigma}}w=\sup_{\partial M_{\Sigma}}w\leq0
\]
and, therefore,%
\[
H\sup_{M}u-1\leq\sup_{M_{\Sigma}}w\leq0.
\]
This shows that $u\leq1/H$. To conclude the proof, observe that, by
(\ref{lap-meancurv}), $u\in C^{0}\left(  M_{\Sigma}\right)  \cap C^{\infty
}\left(  \mathrm{int}M_{\Sigma}\right)  $ is a superharmonic function.
Moreover, by assumption, $u$ is bounded and $u=0$ on $\partial M_{\Sigma}$.
Therefore, using again Theorem \ref{th_ahlfors-wholeM} in the form of a
minimum principle, we deduce%
\[
\inf_{M_{\Sigma}}u=\inf_{\partial M_{\Sigma}}u=0,
\]
proving that $u\geq0$.
\end{proof}

\begin{remark}
It is well known that, in case $\partial M=\emptyset$, the above volume growth
assumption implies that the vertically bounded $H$-graph must be necessarily
minimal, $H=0$. Actually, according to Theorem 5.1 in \cite{RS-Revista}, the
same conclusion holds if $\mathrm{vol}B_{R}\leq C_{1}e^{C_{2}R^{2}}$ for some
constants $C_{1},C_{2}>0$. Indeed, under this condition, the weak
maximum/minimum principle at infinity for the mean-curvature operator holds on
$M$. Therefore, there exists a sequence $x_{k}$ along which%
\[%
\begin{array}
[c]{ll}%
(a) & u\left(  x_{k}\right)  <\inf_{M}u+1/k\ \\
(b) & mH\equiv-\operatorname{div}((1+\left\vert \nabla_{M}u\right\vert
^{2}\left(  x_{k}\right)  )^{-1/2}\nabla_{M}u\left(  x_{k}\right)  )<1/k.
\end{array}
\]
This shows that $H\leq0$. In a similar fashion we obtain the opposite
inequality, proving that $H\equiv0$. The same conclusion was also obtained in
\cite{PRS-Pacific} by different methods.

On the other hand, if $\partial M=\emptyset$ and the volume growth of $M$ is
sub-quadratic then $M$ is parabolic with respect to the mean curvature
operator, \cite{RS-Revista}. Therefore, not only the $H$-graph is minimal, but
it must be a slice of $M\times\mathbb{R}$.
\end{remark}

\begin{remark}
Theorem \ref{th_intro_hest_graph} goes in the direction of generalizing
Theorem 4 in \cite{RoRo} by A. Ros and H. Rosenberg to non-homogeneous
domains. Indeed, assume that $m=2,3,4$ and $\mathrm{Sec}_{N}\geq0$. Then, for
every $\left\vert H\right\vert >0$, an $H$-graph $\Sigma=\Gamma_{u}(M)$ in
$N\times\mathbb{R}$ over a domain $M\subseteq N$, is necessarily bounded;
\cite{RoRo, Ch, ENR}. Furthermore, in case $m=2$, it follows by the
Bishop-Gromov comparison theorem that, if $\mathrm{Sec}_{N}\geq0$, then $N$
has quadratic volume growth, that is
\[
\mathrm{vol}B_{R}^{N}\left(  \bar{x}\right)  \leq\omega_{2} R^{2},
\]
where $\omega_{2}$ denotes the area of the unit ball in $\mathbb{R}^{2}$.
Moreover, if $N$ is complete, $\partial N =\emptyset$, then $\overline{M}$is a
complete parabolic manifold with boundary. Indeed, let $d_{M}$ and $d_{N}$
denote the intrinsic distance functions on $M$ and $N$, respectively. Clearly%
\begin{equation}
d_{M}\geq\left.  d_{N}\right\vert _{M\times M} \label{distances}%
\end{equation}
and $\left(  M,d_{M}\right)  $ is a complete metric space. Indeed, from
(\ref{distances}), any Cauchy sequence $\left\{  x_{k}\right\}  \subset\left(
M,d_{M}\right)  $ is Cauchy in the complete space $\left(  N,d_{N}\right)  $.
It follows that $x_{k}\overset{d_{N}}{\rightarrow}\bar{x}\in N$ as
$k\rightarrow+\infty$. Actually, since $M$ is a closed subset of $\left(
N,d_{N}\right)  $, we have $\bar{x}\in M$. To conclude that $x_{k}%
\overset{d_{M}}{\rightarrow}\bar{x}$, simply recall that the metric topology
on $M$ induced by $d_{M}$ is the original topology of $M$, i.e., the subspace
topology inherited from $N$. Moreover, since, by (\ref{distances}),%
\[
\mathrm{vol}B_{R}^{M}\left(  x\right)  \leq\mathrm{vol}\left(  B_{R}%
^{N}\left(  x\right)  \cap M\right)  \leq\mathrm{vol}\left(  B_{R}^{N}\left(
x\right)  \right)  ,
\]
for every $x\in M$, it follows that $M$ enjoys the same volume growth property
of $N$.

In light of the considerations above, Corollary \ref{coro_intro_hest_graph} is
now straightforward.
\end{remark}

We end this section, by considering the more general case of an oriented CMC
hypersurface in the Riemannian product $N\times\mathbb{R}$. Abstracting from
the previous arguments, and up to using more involved computations as in
\cite{AD}, we easily obtain the proof of Theorem \ref{th_intro_hest} stated in
the Introduction.

\begin{proof}
[Proof (of Theorem \ref{th_intro_hest})]Let $f:\Sigma^{m}\rightarrow
N^{m}\times\mathbb{R}$ be a complete, oriented $H$-hypersurface isometrically
immersed in $N\times\mathbb{R}$, and denote by $h$ the projection of the image
of $\Sigma$ on $\mathbb{R}$ under the immersion, that is, $h=\pi_{\mathbb{R}%
}\circ f$. Note that
\begin{equation}
\Delta_{\Sigma}h=n\cos\Theta H\leq0, \label{Lap h}%
\end{equation}
where, we recall, $\Theta\in\lbrack\frac{\pi}{2},\frac{3\pi}{2}]$ stands for
the angle between the Gauss map $\mathcal{N}$ and the vertical vector field
$\partial/\partial t$. Since, by Theorem \ref{th_growth}, $\Sigma$ is
parabolic and $h$ is a bounded below superharmonic function, we can apply the
Ahlfors maximum principle to get
\[
h\geq\inf_{\Sigma}h=\inf_{\partial\Sigma}h=0.
\]
Consider now the function $\varphi$ defined as
\[
\varphi=Hh+\cos\Theta.
\]
We know by Theorem 3.1 in \cite{AD} that $\varphi$ is subharmonic. Since it is
also bounded, applying again the Ahlfors maximum principle we conclude that
\[
Hh-1\leq\varphi\leq\sup_{\Sigma}\varphi=\sup_{\partial\Sigma}\varphi\leq0.
\]
We have thus shown that%
\[
0\leq\pi_{\mathbb{R}}\circ f\left(  x\right)  \leq\frac{1}{H},
\]
as required.
\end{proof}

\section{The $L^{2}$-Stokes theorem \& slice-type
results\label{section-stokes}}

In this section we prove the global divergence theorem stated in the
Introduction as Theorem \ref{th_intro_Stokes}. We also provide a somewhat
weaker form of this result which involves differential inequalities of the
type $\operatorname{div}X\geq f$; see Proposition \ref{propineq} below. This
latter, together with the Ahlfors maximum principle, is then applied to prove
slice-type results for hypersurfaces in product spaces and for graphs; see
Theorems \ref{th_intro_slice} and \ref{th_intro_slice_graphs} in the
Introduction. Actually, the graph-version of this result also requires a
Liouville-type theorem for the mean curvature operator on manifolds with
boundary, under volume growth conditions. This is modeled on \cite{RS-Revista}.

\subsection{Global divergence theorems\label{subsection-divergence}}

Recall that, for a given smooth, compactly supported vector field $X$ on an
oriented Riemannian manifold $M$ with boundary $\partial M\neq\emptyset$, the
ordinary Stokes theorem asserts that
\begin{equation}
\int_{M}\operatorname{div}X=\int_{\partial M}\langle X,\nu\rangle,
\label{stokes}%
\end{equation}
where $\nu$ is the exterior unit normal to $\partial M$. In particular, this
holds for every smooth vector field if $M$ is compact. The result still holds
if we relax the regularity conditions on $X$ up to interpret its divergence in
the sense of distributions. To be precise, we introduce the following definition.

\begin{definition}
\label{def_weakdiv}Let $X$ be a vector field on $M$ satisfying $X\in
L_{loc}^{1}(M)$ and $\left\langle X,\nu\right\rangle \in L_{loc}^{1}\left(
\partial M\right)  $. The \textit{distributional divergence} of $X$ is defined
by
\begin{equation}
(\operatorname{div}X,\varphi)=-\int_{M}\langle X,\nabla\varphi\rangle
+\int_{\partial M}\varphi\langle X,\nu\rangle, \label{weakdiv}%
\end{equation}
for every $\varphi\in C_{c}^{\infty}(M)$.
\end{definition}

\begin{remark}
\label{rem_divweak}The above definition extends trivially to $\varphi\in
Lip_{c}\left(  M\right)  $. Actually, more is true. Recall that, given a
domain $D\subseteq M$, $W_{0}^{1,p}\left(  D\right)  $ denotes the closure of
$C_{c}^{\infty}(D)$ in $W^{1,p}(D)$. Then, by a density argument, the previous
definition extends to every $\varphi\in C_{c}^{0}\left(  M\right)  \cap
W_{0}^{1,2}\left(  M\right)  $. Indeed, let $\varphi$ be such a function.
Then, we find an approximating sequence $\varphi_{n}\in C_{c}^{\infty}\left(
M\right)  $ such that $\varphi_{n}\rightarrow\varphi$ in $W^{1,2}\left(
M\right)  $, as $n\rightarrow+\infty$. Since $\mathrm{supp}\left(
\varphi\right)  $ is compact, we can assume that there exists a domain
$\Omega\subset\subset M$ such that $\mathrm{supp}\left(  \varphi_{n}\right)
\subset\Omega$, for every $n$. Moreover, a subsequence (still denoted by
$\varphi_{n}$) converges pointwise a.e. to $\varphi$. Let $c=\max
_{M}\left\vert \varphi\right\vert +1$ and define $\phi_{n}=f\circ\varphi
_{n}\in Lip_{c}\left(  M\right)  $ where%
\[
f\left(  t\right)  =\left\{
\begin{array}
[c]{ll}%
c, & t\geq c\\
t, & -c<t<c\\
-c, & t\leq-c.
\end{array}
\right.
\]
Note that $\left\{  \phi_{n}\right\}  $ is an equibounded sequence,
$\mathrm{supp}\left(  \phi_{n}\right)  \subset\Omega$ and, furthermore,
$\phi_{n}\rightarrow f\circ\varphi=\varphi$ in $W^{1,2}\left(  M\right)  $ and
pointwise a.e. in $M$. Therefore, evaluating (\ref{weakdiv}) along $\phi_{n}$,
taking limits as $n\rightarrow+\infty$ and using the dominated convergence
theorem completes the proof.
\end{remark}

Now, suppose also that $\operatorname{div}X\in L_{loc}^{1}(M)$. Then we can
write%
\[
(\operatorname{div}X,\varphi)=\int_{M}\varphi\operatorname{div}X
\]
and, therefore, from (\ref{weakdiv}) we get%
\[
\int_{M}\varphi\operatorname{div}X=-\int_{M}\langle X,\nabla\varphi
\rangle+\int_{\partial M}\varphi\langle X,\nu\rangle.
\]
In particular, if $X$ is compactly supported , by choosing $\varphi=1$ on the
support of $X$, we recover the Stokes formula (\ref{stokes}) for every
compactly supported vector field $X$ satisfying $X\in L_{loc}^{1}(M)$,
$\operatorname{div}X\in L_{loc}^{1}\left(  {M}\right)  $ and $\left\langle
X,\nu\right\rangle \in L_{loc}^{1}\left(  \partial M\right)  $.\bigskip

Note that, by similar reasonings, if the vector field $X\in L_{loc}^{1}\left(
M\right)  $ has a weak divergence $\operatorname{div}X\in L_{loc}^{1}\left(
M\right)  $ and $\left\langle X,\nu\right\rangle \in L_{loc}^{1}\left(
\partial M\right)  ,$ then, for every $\rho\in C_{c}^{0}\left(  M\right)  \cap
W_{0}^{1,2}\left(  M\right)  $, we have that $\operatorname{div}\left(  \rho
X\right)  \in L_{loc}^{1}\left(  M\right)  $. Moreover, as in the smooth case,%
\[
\int_{M}\operatorname{div}\left(  \rho X\right)  =\int_{M}\left\langle
\nabla\rho,X\right\rangle +\int_{M}\rho\operatorname{div}X.
\]
To see this, we take $\varphi\in C_{c}^{\infty}\left(  M\right)  $ and, using
(\ref{weakdiv}) in the form of Remark \ref{rem_divweak}, we compute%
\begin{align*}
\left(  \operatorname{div}\left(  \rho X\right)  ,\varphi\right)   &
=-\int_{M}\left\langle \rho X,\nabla\varphi\right\rangle +\int_{\partial
M}\rho\varphi\left\langle X,\nu\right\rangle \\
&  =-\int_{M}\left\langle X,\nabla\left(  \rho\varphi\right)  \right\rangle
+\int_{\partial M}\rho\varphi\left\langle X,\nu\right\rangle +\int_{M}%
\varphi\left\langle X,\nabla\rho\right\rangle \\
&  =\left(  \operatorname{div}X,\rho\varphi\right)  +\int_{M}\varphi
\left\langle X,\nabla\rho\right\rangle \\
&  =\int_{M}\left(  \rho\operatorname{div}X+\left\langle X,\nabla
\rho\right\rangle \right)  \varphi\\
&  =\left(  \rho\operatorname{div}X+\left\langle X,\nabla\rho\right\rangle
,\varphi\right)  .
\end{align*}
Whence, we conclude that%
\[
\operatorname{div}\left(  \rho X\right)  =\rho\operatorname{div}X+\left\langle
X,\nabla\rho\right\rangle \in L_{loc}^{1}\left(  M\right)
\]
as desired.

All these facts will be tacitly employed several times in the rest of the
Section.\bigskip

If $M$ is not compact, we can still prove a global version of Stokes theorem
for vector fields with prescribed asymptotic behavior at infinity. This is the
content of Theorem \ref{th_intro_Stokes}.\newline

\begin{proof}
[Proof (of Theorem \ref{th_intro_Stokes})]Suppose $M$ is parabolic. According
to Theorem~\ref{capacity_parabolicity} (ii) there exists an increasing
sequence of functions $\varphi_{n}\in C_{c}(M) \cap W^{1,2}(M)$ such that
$0\leq\varphi_{n}\leq1$ and
\[
\varphi_{n} \to1\,\, \text{ locally uniformly on }\, M \,\,\text{ and }
\int_{M} |\nabla\varphi_{n}|^{2} \to0.
\]
Consider now any vector field $X$ satisfying (\ref{KNR1}). Since $\varphi
_{n}X$ is compactly supported, applying the usual (weak) divergence theorem we
get%
\begin{equation}
\int_{M}\operatorname{div}\left(  \varphi_{n}X\right)  =\int_{\Omega_{n}%
}\operatorname{div}\left(  \varphi_{n}X\right)  =\int_{\partial_{1}\Omega_{n}%
}\varphi_{n}\left\langle X,\nu\right\rangle . \label{KNR2}%
\end{equation}
On the other hand%
\[
\int_{M}\operatorname{div}\left(  \varphi_{n}X\right)  =\int_{M}\left\langle
\nabla\varphi_{n},X\right\rangle +\int_{M}\varphi_{n}\operatorname{div}X,
\]
where%
\[
\left\vert \int_{M}\left\langle \nabla\varphi_{n},X\right\rangle \right\vert
\leq\left(  \int_{M}\left\vert \nabla\varphi_{n}\right\vert ^{2}\right)
^{\frac{1}{2}}\left(  \int_{M}\left\vert X\right\vert ^{2}\right)  ^{\frac
{1}{2}}\rightarrow0
\]
as $n\rightarrow+\infty$. Moreover%
\[
\int_{M}\varphi_{n}\operatorname{div}X=\int_{M}\varphi_{n}(\operatorname{div}%
X)_{+}-\int_{M}\varphi_{n}(\operatorname{div}X)_{-}%
\]
and%
\[
\int_{M}\varphi_{n}(\operatorname{div}X)_{+}\leq\int_{M}\varphi_{n}%
(\operatorname{div}X)_{-}+\int_{\partial_{1}\Omega_{n}}\varphi_{n}\left\langle
X,\nu\right\rangle -\int_{M}\left\langle \nabla\varphi_{n},X\right\rangle .
\]
Using the monotone convergence theorem and the fact that $0\leq\varphi_{n}%
\leq1$, we obtain
\[
\int_{M}(\operatorname{div}X)_{+}\leq\int_{M}(\operatorname{div}X)_{-}%
+\int_{\partial_{1}\Omega_{n}}\varphi_{n}\left\langle X,\nu\right\rangle
<+\infty.
\]
Hence $\operatorname{div}X\in L^{1}(M)$ and taking limits on both sides of
(\ref{KNR2}) completes the first part of the proof.

Conversely, assume that $M$ is not parabolic so that $M$ possesses a smooth,
finite, positive Green kernel, \cite{Gr1, GN}. We shall show that the global
Stokes theorem fails. To this end, choose an exhaustion $\left\{  \Omega
_{n}\right\}  $ of $M$ by smooth and relatively compact domains. Then, the
Neumann Green kernel $G\left(  x,y\right)  $ of $M$ is obtained as the limit
of the Green functions $G_{n}\left(  x,y\right)  $ of $\Omega_{n}$ which
satisfy%
\[
\left\{
\begin{array}
[c]{ll}%
\Delta G_{n}\left(  x,y\right)  =-\delta_{x}\left(  y\right)  & \text{on
}\Omega_{n}\cap\mathrm{int}M\\
\dfrac{\partial G_{n}}{\partial\nu}=0 & \text{on }\partial_{1}\Omega_{n}\\
G_{n}=0 & \text{on }\partial_{0}\Omega_{n}.
\end{array}
\right.
\]
Let $f\geq0$ be a smooth function compactly supported in $\mathrm{int}M$. For
each $n$ define%
\[
u_{n}\left(  x\right)  =\int_{\Omega_{n}}G_{n}\left(  x,y\right)  f\left(
y\right)  dy.
\]
Then, each $u_{n}$ is a positive, classical solution of the boundary value
problem%
\[
\left\{
\begin{array}
[c]{ll}%
\Delta u_{n}=-f & \text{on }\Omega_{n}\cap\mathrm{int}M\\
\dfrac{\partial u_{n}}{\partial\nu}=0 & \text{on }\partial_{1}\Omega_{n}\\
u_{n}=0 & \text{on }\partial_{0}\Omega_{n}.
\end{array}
\right.
\]
By the maximum principle and the boundary point lemma, the sequence is
monotonically {increasing} and converges to a solution $u$ of%
\[
\left\{
\begin{array}
[c]{ll}%
\Delta u=-f & \text{on }M\\
\dfrac{\partial u}{\partial\nu}=0 & \text{on }\partial M.
\end{array}
\right.
\]
Also, using Fatou Lemma,%
\[
\int_{M}\left\vert \nabla u_{n}\right\vert ^{2}\geq\int_{M}\left\vert \nabla
u\right\vert ^{2}.
\]
Now consider the vector field%
\[
X=\nabla u.
\]
Clearly $X$ satisfies all the conditions in (\ref{KNR1}). On the other hand,
we have%
\[
\int_{M}\operatorname{div}X=-\int_{M}f\neq0
\]
and%
\[
\int_{\partial M}\left\langle X,\nu\right\rangle =\int_{\partial M}%
\dfrac{\partial u}{\partial\nu}=0,
\]
proving that the global Stokes theorem fails to hold.
\end{proof}

Using Definition \ref{def_weakdiv} of weak divergence one could introduce the
notion of weak solution of a differential inequality like $\operatorname{div}%
X\geq f$. We stress that $\operatorname{div}X$ is not required to be a function.

\begin{definition}
\label{def_weak-sol-divX}Let $X\in L_{loc}^{1}\left(  M\right)  $ be a vector
field satisfying $\left\langle X,\nu\right\rangle \in L_{loc}^{1}\left(
\partial M\right)  $ and let $f\in L_{loc}^{1}\left(  M\right)  $. We say that
$\operatorname{div}X\geq f$ in the distributional sense on $M$ if%
\[
\left(  \operatorname{div}X,\varphi\right)  \geq\int_{M}f\varphi,
\]
for every $0\leq\varphi\in C_{c}^{\infty}\left(  M\right)  $. Actually,
according to Remark \ref{rem_divweak}, the definition extends to every
$0\leq\varphi\in C_{c}^{0}\left(  M\right)  \cap W^{1,2}\left(  M\right)  $.

In the special case where $f=0$ and $X=\nabla u$ for some $u\in W_{loc}%
^{1,2}\left(  M\right)  $ satisfying $\partial u/\partial\nu\in L_{loc}%
^{1}\left(  \partial M\right)  $, we obtain the corresponding notion of weak
solution of $\Delta u\geq0$ on $M$.
\end{definition}

Although elementary, it is important to realize that, as in the smooth
setting, the above definition is compatible with that of weak Neumann
\ subsolution given in the Introduction.

\begin{lemma}
\label{lemma_equiv-weak-def}Let $u\in W_{loc}^{1,2}\left(  M\right)  $ satisfy
$\partial u/\partial\nu\in L_{loc}^{1}\left(  \partial M\right)  $. Then $u$
is a weak Neumann subsolution of the Laplace equation provided $u$ satisfies%
\[
\left\{
\begin{array}
[c]{ll}%
\Delta u\geq0 & \text{on }M\medskip\\
\dfrac{\partial u}{\partial\nu}\leq0 & \text{on }\partial M,
\end{array}
\right.
\]
where the differential inequality is interpreted according to Definition
\ref{def_weak-sol-divX}.
\end{lemma}

\begin{proof}
Straightforward from the equation%
\[
\left(  \Delta u,\varphi\right)  \overset{\mathrm{def}}{=}-\int_{M}%
\left\langle \nabla u,\nabla\varphi\right\rangle +\int_{\partial M}%
\dfrac{\partial u}{\partial\nu}\varphi,
\]
with $0\leq\varphi\in C_{c}^{\infty}\left(  M\right)  $.
\end{proof}

Reasoning as in the proof of Theorem \ref{th_intro_Stokes}, we can now prove
the following result which extends to manifolds with boundary a result in
\cite{HPV}.

\begin{proposition}
\label{propineq} Let $\left(  M,g\right)  $ be an $m$-dimensional, parabolic
manifold with smooth boundary $\partial M$. Let $X$ be a vector field on $M$
satisfying:%
\[
\text{(a) }\left\vert X\right\vert \in L^{2}\left(  M\right)  \text{; (b)
}0\geq\left\langle X,\nu\right\rangle \in L_{loc}^{1}\left(  \partial
M\right)  .
\]
Assume that $\operatorname{div}X\geq f$ for some $f\in L^{1}(M)$ in the sense
of distributions. Then
\[
\int_{M}f\leq\int_{\partial M}\langle X,\nu\rangle.
\]
{The same conclusion holds if $0\leq f\in L^{1}_{loc}(M)$ and yields
\[
f\equiv0.
\]
} Moreover, if $\operatorname{div}X\geq0$ in the distributional sense, then%
\[
\int_{M}\langle X,\nabla\alpha\rangle\leq\int_{\partial M}\alpha\left\langle
X,\nu\right\rangle
\]
for every $0\leq\alpha\in C_{c}^{\infty}\left(  M\right)  $.
\end{proposition}

\begin{proof}
Choose a smooth, relatively compact exhaustion $\Omega_{n}\subset M$ and
denote by $\varphi_{n}$ the equilibrium potential of the capacitor
$(\overline{\Omega}_{0},\Omega_{n})$. Extend $\varphi_{n}$ to be identically
$1$ on $\Omega_{0}$ and identically $0$ on $M\backslash\Omega_{n}$. Then, by
assumption,%
\begin{align*}
\int_{M}\varphi_{n}f  &  \leq\left(  \operatorname{div}X,\varphi_{n}\right) \\
&  =-\int_{M}\langle X,\nabla\varphi_{n}\rangle+\int_{\partial M}\varphi
_{n}\langle X,\nu\rangle\\
&  \leq\left(  \int_{M}|X|^{2}\right)  ^{\frac{1}{2}}\left(  \int_{M}%
|\nabla\varphi_{n}|^{2}\right)  ^{\frac{1}{2}}+\int_{\partial M}\varphi
_{n}\langle X,\nu\rangle.
\end{align*}
The first part of the statement follows by taking the $\limsup$ as
$n\rightarrow+\infty$ and applying the Fatou Lemma and either the monotone
convergence theorem if $0\leq f\in L^{1}_{loc}(M)$ or the dominated
convergence theorem if $f\in L^{1}(M)$. For what concern the second part,
consider the test function $\eta=\varphi_{n}\alpha$. Then,
\begin{align*}
0  &  \leq\left(  \operatorname{div}X,\alpha\varphi_{n}\right) \\
&  =-\int_{M}\alpha\langle X,\nabla\varphi_{n}\rangle-\int_{M}\varphi
_{n}\langle X,\nabla\alpha\rangle+\int_{\partial M}\alpha\varphi_{n}\langle
X,\nu\rangle\\
&  \leq\sup_{M}|\alpha|\left(  \int_{M}|X|^{2}\right)  ^{\frac{1}{2}}\left(
\int_{M}|\nabla\varphi_{n}|^{2}\right)  ^{\frac{1}{2}}-\int_{M}\varphi
_{n}\langle X,\nabla\alpha\rangle+\int_{\partial M}\alpha\varphi_{n}\langle
X,\nu\rangle.
\end{align*}
and the conclusion follows as above computing the $\limsup$ as $n\rightarrow
+\infty$.
\end{proof}

\subsection{Slice-type theorems for hypersurfaces in a half-space}

This Section is devoted to the proofs of Theorems \ref{th_intro_slice} and
\ref{th_intro_slice_graphs} stated in the Introduction. The first one of these
results involves a complete hypersurface\ $\Sigma$ contained in the half-space
$N\times\lbrack0+\infty)$ of the ambient product space $N\times\mathbb{R}$. It
is assumed that the boundary $\partial\Sigma\neq\emptyset$ lies in the slice
$N\times\left\{  0\right\}  $ and that $\Sigma$ has non-positive mean
curvature $H\leq0$ with respect to the \textquotedblleft
downward\textquotedblright Gauss map. The result states that, under a
quadratic area growth assumption on $\Sigma$ and regardless of the geometry of
$N$, the portion of the hypersurface $\Sigma$ in any upper-halfspace of
$N\times\mathbb{R}$ must have infinite volume unless $\Sigma$ is contained in
the totally geodesic slice $N\times\left\{  0\right\}  $. The second result
provides a graphical version of this theorem when $\Sigma=\Gamma_{u}\left(
M\right)  $. If $M$ satisfies a quadratic volume growth assumption, then each
superlevel set $M_{t}=\left\{  u\geq t>0\right\}  \subseteq M$ has infinite
volume unless $\Sigma$ is contained in the totally geodesic slice
$M\times\left\{  0\right\}  $. Note that $M_{t}$ is the orthogonal projection
of $\Sigma\cap\lbrack t,+\infty)$ on the slice $M\times\left\{  0\right\}
$.\bigskip

Let us begin with the

\begin{proof}
[Proof (of Theorem \ref{th_intro_slice})]Suppose that $\Sigma$ is not
contained in the slice $N\times\left\{  0\right\}  $. If the height function
$h$ on $\Sigma$ is bounded from above (for the precise definition of $h$ see
the proof of Theorem \ref{th_intro_hest} in Subsection
\ref{section-heightestimates}) the parabolicity of $\Sigma$ in the form of the
Ahlfors maximum principle implies that
\[
h\leq\sup_{\Sigma}h=\sup_{\partial\Sigma}h=0.
\]
The conclusion is then immediate because, by assumption, $\Sigma$ is contained
in the half-space $N\times\lbrack0,+\infty)$. Suppose now that $\sup_{\Sigma
}h=+\infty$, so that $\Sigma\cap N\times\left\{  t\right\}  \neq\emptyset$ for
an arbitrary $t>0$. Letting%
\[
\Sigma_{t}=\Sigma\cap N\times\lbrack t,+\infty)=\left\{  p\in\Sigma:h\left(
p\right)  \geq t\right\}  ,
\]
and since $\mathrm{vol}\left(  \Sigma_{t}\right)  \geq\mathrm{vol}\left(
\Sigma_{s}\right)  $, for every $s\geq t$, we can assume that $\mathrm{vol}%
\left(  \Sigma_{t}\right)  <+\infty$ for every $t>>1$. Moreover, by Sard
theorem we can suppose that $t$ is a regular value of $\left.  h\right\vert
_{\mathrm{int}\Sigma}$. In particular, $\Sigma_{t}$ is a smooth complete
hypersurface with boundary $\partial\Sigma_{t}=\left\{  p\in\Sigma:h\left(
p\right)  =t\right\}  $ and exterior unit normal $\nu_{t}=-\nabla h/|\nabla
h|$. Clearly, $\Sigma_{t}$ is parabolic because it has finite volume.
According to (\ref{Lap h}), $h$ is a subharmonic function on $\Sigma_{t}$ and
satisfies $\left\vert \nabla h\right\vert \leq1$. In particular, $\left\vert
\nabla h\right\vert \in L^{2}\left(  \Sigma_{t}\right)  $. For any
$\varepsilon>0$ define%
\[
h_{\varepsilon}=\max\left\{  h,t+\varepsilon\right\}  .
\]
Then $h_{\varepsilon}$ is again subharmonic on $M_{t}$, it has finite Dirichet
energy $\left\vert \nabla h_{\varepsilon}\right\vert \in L^{2}\left(
\Sigma_{t}\right)  $ and, furthermore, $\partial h_{\varepsilon}/\partial
\nu=0$ on $\partial\Sigma_{t}$. Therefore, we can apply Proposition
\ref{propineq} and deduce that $h_{\varepsilon}$ has to be harmonic on
$\Sigma_{t}$. Actually, since $h_{\varepsilon}$ is bounded from below on the
parabolic manifold $\Sigma_{t}$ it follows that $h_{\varepsilon}$ is constant
on every connected component of $\Sigma_{t}$. Whence, on noting that
$h_{\varepsilon}=t+\varepsilon$ on $\partial\Sigma_{t}$ we obtain that $t\leq
h\leq t+\varepsilon$ on $\Sigma_{t}$. Since this holds for every
$\varepsilon>0$ we conclude that $h\equiv t$ on $\Sigma_{t}$, contradicting
the assumption of $h$ being unbounded.
\end{proof}

The proof of Theorem \ref{th_intro_slice_graphs} is completely similar but
requires some preparation. The next Liouville-type result for the mean
curvature operator is adapted from \cite{RS-Revista}; see also \cite{CY-CPAM,
Ch-Manuscripta}. We provide a detailed proof for the sake of completeness.

\begin{theorem}
\label{th_areagrowth_meancurvop}Let $(M,g)$ be a complete Riemannian manifold
with boundary $\partial M\neq\emptyset$. If, for some reference point
$o\in\mathrm{int}M$,%
\begin{equation}
\frac{1}{{\mathrm{Area}}\left(  {\partial}_{0}{B_{R}}\left(  o\right)
\right)  }\notin L^{1}(+\infty), \label{area-varphipar}%
\end{equation}
then the following holds. Let $u\in C^{1}\left(  M\right)  $ be a weak Neumann
solution of the problem%
\begin{equation}
\left\{
\begin{array}
[c]{ll}%
\operatorname{div}\left(  \dfrac{\nabla u}{\sqrt{1+\left\vert \nabla
u\right\vert ^{2}}}\right)  \geq0 & \text{on }M\medskip\\
\dfrac{\partial u}{\partial\nu}\leq0 & \text{on }\partial M\medskip\\
\sup_{M}u<+\infty. &
\end{array}
\right.  \label{parab-meancurv}%
\end{equation}
Then $u\equiv\mathrm{const}.$
\end{theorem}

\begin{remark}
As already pointed out for the Laplace-Beltrami operator, being a weak Neumann
solution of $\operatorname{div}((1+\left\vert \nabla u\right\vert ^{2}%
)^{-1/2}\nabla u))\geq0$ means that%
\begin{equation}
-{\displaystyle\int_{M}} \left\langle \dfrac{\nabla u}{\sqrt{1+\left\vert
\nabla u\right\vert ^{2}}},\nabla\varphi\right\rangle \geq0,
\label{weak_sub_meancurv}%
\end{equation}
for every $0\leq\varphi\in C_{c}^{\infty}\left(  M\right)  $. Actually, it is
obvious that the same definition extends to any elliptic operator of the form
$L_{\Phi}\left(  u\right)  =\operatorname{div}(\Phi(\left\vert \nabla
u\right\vert )\nabla u)$, where $\Phi\left(  t\right)  $ is subjected to
certain structural conditions. Moreover, under the assumption%
\[
|\nabla u| \in L_{loc}^{1}\left(  \partial M\right)  ,
\]
this definition is also coherent with the notion of weak divergence. Namely
$u$ satisfies (\ref{weak_sub_meancurv}) provided $\left(  \operatorname{div}%
X,\varphi\right)  \geq0$ and $\partial u/\partial\nu\leq0$, where we have set
$X=(1+\left\vert \nabla u\right\vert ^{2})^{-1/2}\nabla u$. This follows
immediately from the equation%
\[
\left(  \operatorname{div}X,\varphi\right)  \overset{\mathrm{def}}%
{=}-{\displaystyle\int_{M}} \left\langle \dfrac{\nabla u}{\sqrt{1+\left\vert
\nabla u\right\vert ^{2}}},\nabla\varphi\right\rangle +\int_{\partial M}%
\frac{\varphi}{\sqrt{1+\left\vert \nabla u\right\vert ^{2}}}\frac{\partial
u}{\partial\nu}.
\]

\end{remark}

\begin{remark}
\label{rmk_areagrowth_meancurvop} If we take $\Phi\left(  t\right)  =1$ in the
argument below we recover Theorem \ref{th_growth} by Grigor'yan, in the form
of a Liouville result for $C^{1}(M)$ subsolutions of the Laplace equation.
\end{remark}

\begin{proof}
Let $u$ be as in the statement of the theorem and assume, by contradiction,
that $u$ is non-constant on the ball $B_{R_{0}}(o)$, for some $R_{0}>0$.
Without loss of generality we can suppose that $u\leq0$ on $M$. Define%
\[
\Phi\left(  t\right)  =\frac{1}{\sqrt{1+t^{2}}}.
\]
Now, having fixed $R>R_{0}$ and $\varepsilon>0$, we choose $\rho
=\rho_{\varepsilon,R}$ as follows:%
\[
\rho(x)=%
\begin{cases}
1 & \mathrm{on}\quad B_{R}(o)\\
\frac{R+\varepsilon-r(x)}{\varepsilon} & \mathrm{on}\quad B_{R+\varepsilon
}(o)\backslash B_{R}(o)\\
0 & \mathrm{elsewhere}.
\end{cases}
\]
Inserting the test function $\varphi=\rho\mathrm{e}^{u}$ into
(\ref{weak_sub_meancurv}) and elaborating we get%
\begin{align*}
0  &  \leq-\int_{M}\langle\Phi(|\nabla u|)\nabla u,\nabla(\rho\mathrm{e}%
^{u})\rangle\\
&  =-\int_{M}\mathrm{e}^{u}\Phi(|\nabla u|)\left\langle \nabla u,\nabla
\rho\right\rangle -\int_{M}\rho\mathrm{e}^{u}\Phi(|\nabla u|)\left\vert \nabla
u\right\vert ^{2}.
\end{align*}
Then, on noting also that $\partial M$ has measure zero, we have%
\[
\varepsilon^{-1}\int_{(B_{R+\varepsilon}(o)\backslash B_{R}(o))\cap
\mathrm{int}M}\mathrm{e}^{u}\Phi(|\nabla u|)\langle\nabla u,\nabla
r\rangle\geq\int_{B_{R}(o)\cap\mathrm{int}M}\mathrm{e}^{u}\Phi(|\nabla
u|)\left\vert \nabla u\right\vert ^{2}.
\]
Using the co-area formula and letting $\varepsilon\rightarrow0$ we get, for
a.e. $R>R_{0}$,%
\[
\int_{\partial_{0}B_{R}\left(  o\right)  }\mathrm{e}^{u}\Phi(|\nabla
u|)\langle\nabla u,\nabla r\rangle\geq\int_{B_{R}(o)\cap\mathrm{intM}%
}\mathrm{e}^{u}\Phi(|\nabla u|)|\nabla u|^{2}.
\]
On the other hand, using the Cauchy-Schwartz and H\"{o}lder inequalities, we
obtain%
\begin{align*}
\int_{\partial_{0}B_{R}\left(  o\right)  }\!\!\!\!\mathrm{e}^{u}\Phi(|\nabla
u|)\langle\nabla u,\nabla r\rangle &  \leq\int_{\partial_{0}B_{R}\left(
o\right)  }\mathrm{e}^{u}\Phi(|\nabla u|)\left\vert \nabla u\right\vert \\
&  \leq\left(  \int_{\partial_{0}B_{R}\left(  o\right)  }\!\!\!\!
\mathrm{e}^{u}\Phi(|\nabla u|)\right)  ^{\frac{1}{2}}\left(  \int
_{\partial_{0}B_{R}\left(  o\right)  }\!\!\!\! \mathrm{e}^{u}\Phi(|\nabla
u|)\left\vert \nabla u\right\vert ^{2}\right)  ^{\frac{1}{2}}\\
&  \leq\mathrm{Area}(\partial_{0}B_{R}(o))^{\frac{1}{2}}\left(  \int
_{\partial_{0}B_{R}\left(  o\right)  }\mathrm{e}^{u}\Phi(|\nabla u|)\left\vert
\nabla u\right\vert ^{2}\right)  ^{\frac{1}{2}}.
\end{align*}
Now, set
\[
H(R)=\int_{B_{R}(o)\cap\mathrm{intM}}\mathrm{e}^{u}\Phi(|\nabla u|)\left\vert
\nabla u\right\vert ^{2},
\]
Then, by the co-area formula and the previous inequalities,
\[
\frac{H^{\prime}(R)}{H(R)^{2}}\geq\frac{1}{\mathrm{Area}(\partial_{0}%
B_{R}(o))}.
\]
Integrating this latter on $[R_{0},R]$ and letting $R\rightarrow+\infty$ we
conclude%
\[
H(R_{0})\leq\frac{1}{\int_{R_{0}}^{+\infty}\mathrm{Area}(\partial_{0}%
B_{R}(o))^{-1}}=0,
\]
proving that%
\[
\int_{B_{R_{0}}(o)\cap\mathrm{intM}}\mathrm{e}^{u}\Phi(|\nabla u|)\left\vert
\nabla u\right\vert ^{2}=0.
\]
Therefore, $u$ must be constant on $B_{R_{0}}(o)$, leading to a contradiction.
\end{proof}

We are now ready to prove the slice theorem for graphs.

\begin{proof}
[Proof (of Theorem \ref{th_intro_slice_graphs})]Let $\Sigma=\Gamma_{u}\left(
M\right)  $, with $u\in C^{0}\left(  M\right)  \cap C^{\infty}\left(
\mathrm{int}M\right)  $, and for every $s\in\mathbb{R}$ define%
\[
M_{s}:=\{x\in M:u(x)\geq s\}.
\]
By the assumption on $\partial\Sigma=\Gamma_{u}\left(  \partial M\right)  $,
there exists $t>0$ such that $M_{t}\subset\mathrm{int}M$ and $\mathrm{vol}%
(M_{t})<+\infty$. Assume that $M_{t}\neq\emptyset$ for, otherwise, as in
Theorem \ref{th_intro_slice}, the proof is easier. We claim that $u$ is
constant on $M_{t}$. Indeed, by contradiction, suppose that this is not the
case. Then, by Sard Theorem, we can choose $t<c<\sup_{M}u$ such that $c$ is a
regular value of $\left.  u\right\vert _{\mathrm{int}M}$. Thus, the closed
subset $M_{c}$ is a complete manifold with boundary $\partial M_{c}%
\neq\emptyset$ and exterior unit normal $\nu_{c}=-\nabla u/|\nabla u|$. In
particular, as a complete manifold with finite volume, $M_{c}$ is parabolic.
Since the smooth function $u$ satisfies%
\[
\operatorname{div}\left(  \dfrac{\nabla_{M}u}{\sqrt{1+|\nabla_{M}u|^{2}}%
}\right)  =-mH\geq0\text{, on }M_{c}%
\]
then, having fixed any $\varepsilon>0$, the same differential inequality holds
for%
\[
u_{\varepsilon}=\max\left\{  u,c+\varepsilon\right\}  ;
\]
see e.g. \cite{PS-Fortaleza}. Note also that $\partial u_{\varepsilon
}/\partial\nu=0$ on $\partial M_{c}$. Summarizing, the vector field%
\[
X_{\varepsilon}=\dfrac{\nabla_{M}u_{\varepsilon}}{\sqrt{1+|\nabla
_{M}u_{\varepsilon}|^{2}}}%
\]
satisfies
\[
\left\{
\begin{array}
[c]{ll}%
\operatorname{div}_{M}X_{\varepsilon}\geq0 & \text{on }M_{c}\\
1\geq\left\vert X_{\varepsilon}\right\vert \in L^{2}\left(  M_{c}\right)  & \\
{0=\left\langle X_{\varepsilon},\nu_{c}\right\rangle }. &
\end{array}
\right.
\]
By applying Proposition \ref{propineq} we deduce that $\operatorname{div}%
_{M}X=0$ on $M_{c}$, i.e., $\Sigma_{c}=\Gamma_{u}\left(  M_{c}\right)  $ is a
minimal graph. Actually, since $\mathrm{vol}\left(  M_{c}\right)  <+\infty$,
by Theorem \ref{th_areagrowth_meancurvop} we get that $u_{\varepsilon}$ must
be constant on every connected component of $M_{c}$. Since $u_{\varepsilon
}=c+\varepsilon$ on $\partial M_{c}$ it follows that $c\leq u\leq
c+\varepsilon$ on $M_{c}$. Whence, using the fact that $\varepsilon>0$ was
chosen arbitrarily, we conclude that $u\equiv c$ on $M_{c}$. This contradicts
the fact that $c$ is a regular value of $u$, and the claim is proved.

Since $u$ is constant on $M_{t}$ we have that $\sup_{M}u<+\infty$. We now
distinguish three cases.

(a) Suppose that $\partial\Sigma=\partial M\times\left\{  0\right\}  $ and
$\Sigma\subset\lbrack0,+\infty)$. This means that $u\geq0$ with $u=0$ on
$\partial M$. In this case the conclusion $u\equiv0$ follows exactly as in
proof of Theorem \ref{th_intro_slice}.

(b) Suppose that $\Sigma$ is real analytic, i.e., it is described by a real
analytic function $u$. Since $u$ is constant on the open set $\left\{
u<c\right\}  $ we must conclude that $u$ is constant everywhere.

(c) Suppose that $\cos\widehat{\mathcal{N}_{0}\mathcal{N}}\leq0$ on
$\partial\Sigma=\Gamma_{u}\left(  \partial M\right)  $. This means that
$\partial u/\partial\nu\leq0$ on $\partial M$. \ The desired conclusion
follows by a direct application of Theorem \ref{th_areagrowth_meancurvop}.
\end{proof}

The following corollary is a straightforward consequence of the above proof.

\begin{corollary}
Let $\left(  M,g\right)  $ be a complete manifold with boundary $\partial M$
and assume that $\mathrm{vol}M<+\infty$. Let $\Sigma=\Gamma_{u}\left(
M\right)  $ be a graph with non-positive mean curvature $H\left(  x\right)  $
with respect to the downward Gauss map $\mathcal{N}$. Assume also that the
angle $\theta$ between the Gauss map $\mathcal{N}$ of the graph $\Sigma$ and
the Gauss map $\mathcal{N}_{0}=\left(  -\nu,0\right)  $ of $\partial
M\times\left\{  t\right\}  \hookrightarrow M\times\left\{  t\right\}  $
satisfies $\theta\in\lbrack-\frac{\pi}{2},\frac{\pi}{2}]$. Then $\Sigma$ is a
horizontal slice of $M\times\mathbb{R}$.
\end{corollary}

\appendix

\section{Different notions of parabolicity \& some remarks on minimal
graphs\label{appendix-different}}

Let $M$ be a Riemannian manifold without boundary $\partial M=\emptyset$.
Then, from the stochastic viewpoint, $M$ is called parabolic if the Brownian
motion $X_{t}$ on $M$ is \textit{recurrent}, that is $X_{t}$ enters infinitely
many times a fixed compact set with probability $1$. As recorded in the survey
paper \cite{Gr2}, the recurrence of the Brownian motion for manifolds without
boundary can be characterized in terms of fundamental solutions to the Laplace
equation, maximum principles for superharmonic functions, capacities, heat
kernel, Liouville properties for certain Schr\"{o}dinger equations, volume
growth conditions, function theoretic tests (Khas'minskii criterion), $L^{2}%
$-Stokes theorems (Kelvin-Nevanlinna-Royden criterion) and many other
geometric and potential-theoretic properties.

If $M$ has non-empty boundary $\partial M\neq\emptyset$, a quick check at the
literature shows that there are many (non-equivalent) definitions of
parabolicity. The most classical one, which is also the one we have adopted
throughout the paper, was systematically used by A. Grigor'yan starting from
\cite{Gr, Gr1}, and states that $M$ is parabolic provided the reflected
Brownian motion on $M$ is recurrent. This is equivalent to require the
following Liouville-type property, which imposes Neumann-type boundary
conditions on relevant functions. Namely,

\begin{definition}
\label{def-parab-neumann}A Riemannian manifold $M$ with $\partial
M\neq\emptyset$ is $\mathcal{N}$-parabolic if the only solution of the problem%
\begin{equation}
\left\{
\begin{array}
[c]{ll}%
\Delta u\geq0 & \text{on }M\\
\dfrac{\partial u}{\partial\nu}\leq0 & \text{on }\partial M\\
\sup_{M}u<+\infty &
\end{array}
\right.  \label{parab1}%
\end{equation}
is the constant function $u\equiv\sup_{M}u$.
\end{definition}

Most of the geometric and functional-analytic characterizations of
$\mathcal{N}$-parabolicity of manifolds without boundary have already been
extended to the reflected Brownian motion; see \cite{Gr, Gr1, Gr2}. Two
remarkable exceptions were represented by the $L^{2}$-Stokes theorem and the
Ahlfors-type maximum principles, which are some of the main topics of the
present paper.

A second interesting definition can be found in a paper by R. F. De Lima,
\cite{DeLi}, who was interested in maximum principles at infinity for CMC
surfaces. His definition is oriented in the direction of the classical Ahlfors
maximum principle characterization of parabolic manifolds without boundary.
Apparently there was no further research in this direction. Moreover, note
that, a priori, there is no obvious relation between his notion and the
behaviour of the Brownian motion on $M$. Anyway, in the terminology of De
Lima, we have the following

\begin{definition}
\label{def-parab-lima} A Riemannian manifold $M$ is $\mathcal{A}$-parabolic if
for every solution of the problem
\[
\left\{
\begin{array}
[c]{ll}%
\Delta u\geq0 & \text{on }M\\
\sup_{M}u<+\infty &
\end{array}
\right.
\]
it holds
\[
\sup_{M}u=\sup_{\partial M}u.
\]

\end{definition}

As we already observed in Section \ref{subsection-ahlfors}, it is not
difficult to prove that the classical (i.e. Neumann) definition of
parabolicity implies the one introduced by De Lima. Namely,

\begin{proposition}
\label{GimplDL} Assume that $M$ is a $\mathcal{N}$-parabolic manifold with
boundary $\partial M\neq\emptyset$ and let $u$ be a solution of the problem%
\[
\left\{
\begin{array}
[c]{ll}%
\Delta u\geq0 & \text{on }M\\
\sup_{M}u<+\infty. &
\end{array}
\right.
\]
Then%
\[
\sup_{M}u=\sup_{\partial M}u.
\]

\end{proposition}

Finally, a third fruitful definition comes from very recent works in the
theory of minimal surfaces in the Euclidean space, \cite{CKMR, Pe, LP, Ne}.
From the Brownian motion viewpoint, it states that $M$ is parabolic provided
the absorbed Brownian motion is recurrent, i.e., with probability 1 the particle reaches
the boundary (and dies) in a finite time. From a
deterministic viewpoint, this definition involves Dirichlet boundary
conditions on the relevant functions. In this context, a Riemannian manifold
is said to be parabolic if bounded harmonic functions are determined by their
boundary values. This is equivalent to the following

\begin{definition}
\label{def-parab-minim} A Riemannian manifold $M$ is $\mathcal{D}$-parabolic
if the unique solution of the problem%
\[
\left\{
\begin{array}
[c]{ll}%
\Delta u=0 & \text{on }M\\
u=0 & \text{on }\partial M\\
\sup_{M}\left\vert u\right\vert <+\infty &
\end{array}
\right.
\]
is the constant function $u\equiv0$.
\end{definition}

This notion of parabolicity has been used in the theory of minimal surfaces in
$\mathbb{R}^{3}$ because it turned out to be a powerful tool in order to face
the problem of determining which conformal structures are allowed on a minimal
surface subjected to some geometric restrictions on its image.

The notion of $\mathcal{D}$-parabolicity is related to the classical Neumann
one via the Ahlfors maximum principle. Indeed, the following result follows by
applying twice Proposition \ref{GimplDL} to $u$ and to $-u$.

\begin{proposition}
\label{GimplMin} Assume that $M$ is a $\mathcal{N}$-parabolic manifold with
boundary $\partial M\neq\emptyset$ and let $u$ be a solution of the problem%
\[
\left\{
\begin{array}
[c]{ll}%
\Delta u=0 & \text{on }M\\
u=0 & \text{on }\partial M\\
\sup_{M}\left\vert u\right\vert <+\infty &
\end{array}
\right.
\]
Then $u\equiv0$.
\end{proposition}

In the theory of minimal surfaces in the Euclidean space, $\mathcal{D}%
$-parabolicity is not the only global property of surfaces with boundary that
has been studied. Another property of interest is the quadratic area growth
with respect to the extrinsic distance (see \cite{CKMR,Pe,Ne} for more details
and applications of this property). To be more precise, we say that a surface
$M\subset\mathbb{R}^{3}$ has \textit{quadratic area growth} if, for some $C>0$
and $A>0$, one has
\[
\mathrm{vol}(M\cap\{\sqrt{x_{1}^{2}+x_{2}^{2}+x_{3}^{2}}<R\})\leq CR^{2},
\]
for all $R>A$.

The notions of $\mathcal{D}$-parabolicity and quadratic area growth seem to
be, in general, unrelated concepts. For this reason, this global properties
have been studied separetely in the theory of minimal surfaces in
$\mathbb{R}^{3}$. However, according to Proposition \ref{GimplMin}, the volume
condition
\[
\frac{R}{\mathrm{vol}(B_{R}(o))}\notin L^{1}(+\infty),
\]
is sufficient to guarantee that a complete Riemannian manifold $M$ is
$\mathcal{D}$-parabolic. Hence, all the results obtained in this setting under
geometric conditions on the ambient space and exploiting $\mathcal{D}%
$-parabolicity can be obtained imposing a volume growth condition on the
surface instead. Moreover, since the volume of intrinsic balls is dominated by
that of extrinsic balls with the same radius, we conclude also that any
complete (e.g. properly immersed) surface in the Euclidean space with
quadratic area growth is $\mathcal{D}$-parabolic.

To give an example of how this circle of ideas applies we note that it was
conjectured by W. Meeks that any complete (or properly embedded) minimal graph
over a proper subdomain of the plane is $\mathcal{D}$-parabolic. In
\cite{Ne2}, using refined stochastic methods, R. Neel gave a positive answer
to this conjecture. Actually, he was able to prove that for a complete,
embedded minimal surface with boundary whose Gauss image is eventually
contained in a hyperbolic domain of the sphere, the Brownian motion strikes
the boundary almost surely in finite time. However, apparently, no proofs
based on analytic techniques of this fact has appeared yet in literature.

Nevertheless, it was observed by P. Li and J. Wang \cite[Lemma 1]{LW} that
minimal graphs in $\mathbb{R}^{n+1}$ supported on a domain $\Omega
\subset\mathbb{R}^{n}$ have the following (extrinsic) volume growth property%
\[
\mathrm{vol}(M\cap\{\sqrt{x_{1}^{2}+\cdots+x_{n}^{2}}<R\})\leq(n+1)\omega
_{n}R^{n},
\]
where $\omega_{n}$ denotes the volume of the $n$-dimensional unit sphere. In
particular, for a complete minimal graph $M$ in the Euclidean $3$-space,
\[
\mathrm{vol}(B_{R}(o))\leq\mathrm{vol}(M\cap\{\sqrt{x_{1}^{2}+x_{2}^{2}%
+x_{3}^{2}}<R\})\leq3\omega_{2}R^{2},
\]
where $B_{R}(o)$ denotes the geodesic ball in $M$ of radius $R$ centered at a
reference point $o\in\mathrm{int}M$. Hence, complete \ minimal graphs in
$\mathbb{R}^{3}$ have (intrinsic) quadratic volume growth. In view of
Proposition \ref{GimplMin}, we have then proved the following theorem, that
recovers the result by Neel.

\begin{theorem}
Any complete minimal graph in $\mathbb{R}^{3}$ supported on a domain of the
plane is $\mathcal{D}$-parabolic.
\end{theorem}


\bigskip

\begin{thebibliography}{99}                                                                                               %


\bibitem {AD}L. Al\'{\i}as, M. Dajczer, \textit{Constant mean curvature
hypersurfaces in warped product spaces.} Proc. Edinb. Math. Soc. (2)\textbf{
50} (2007), no. 3, 511--526.

\bibitem {Ch-Manuscripta}Q. Chen, \textit{Liouville theorem for harmonic maps
with potential. }Manuscripta Math.\textbf{ 95} (1998), 507--517.

\bibitem {Ch}X. Cheng, \textit{On constant mean curvature hypersurfaces with
finite index.} Arch. Math. (Basel) \textbf{86} (2006), 365--374.

\bibitem {CheRo}X. Cheng, H. Rosenberg, \textit{Embedded positive constant
{$r$}-mean curvature hypersurfaces in $M^{m}\times\mathbf{R}$.} An. Acad.
Brasil. Ci\^enc. \textbf{77} (2005), no. 2, 183--199.

\bibitem {CY-CPAM}S.Y. Cheng, S.T. Yau, \textit{Differential equations on
Riemannian manifolds and their geometric applications. }Comm. Pure Appl. Math.
\textbf{28} (1975), 333--354.

\bibitem {CKMR}P. Collin, R. Kusner, W. H. Meeks, and H. Rosenberg,
\textit{The topology, geometry and conformal structure of properly embedded
minimal surfaces.} J. Differential Geom. \textbf{67} (2004), no. 2, 377--393.

\bibitem {DeLi}R. F. De Lima, \textit{A maximum principle at infinity for
surfaces with constant mean curvature in Euclidean space.} Ann. Global Anal.
Geom. \textbf{20} (2001), no. 4, 325--343.

\bibitem {ENR}M.F. Elbert, B. Nelli, H. Rosenberg, \textit{Stable constant
mean curvature hypersurfaces.} Proc. Amer. Math. Soc. \textbf{135} (2007), 3359--3366.

\bibitem {Gr}A. Grigor'yan, \textit{Existence of the Green function on a
manifold.} Russian Math. Surveys \textbf{38} (1983), 190--191. http://iopscience.iop.org/0036-0279/38/1/L15/pdf/0036-0279\_38\_1\_L15.pdf.

\bibitem {Gr1}A. Grigor'yan, \textit{On the existence of positive fundamental
solutions of the Laplace equation on Riemannian manifolds. \ }Mat. Sb. (N.S.)
\textbf{128 }(1985), no. 3, 354--363. http://m.iopscience.iop.org/0025-5734/56/2/A05/pdf/0025-5734\_56\_2\_A05.pdf.

\bibitem {Gr2}A. Grigor'yan, \textit{Analytic and geometric background of
recurrence and non-explosion of the Brownian motion on Riemannian manifolds.}
Bull. Amer. Math. Soc. (N.S.) 36 (1999), no. 2, 135--249.

\bibitem {GN}A. Grigor'yan, N. Nadirashvili, \textit{The Liouville theorems
and exterior boundary value problems. }Izv. Vyssh. Uchebn. Zaved. Mat.
\textbf{31} (1987), 25--33.

\bibitem {HKM}J. Heinonen, T. Kilpelainen, O. Martio, \textit{Nonlinear
potential theory of degenerate elliptic equations. Second Edition.} Dover
Publications, Inc. Minerola, N.Y. 2006.

\bibitem {He}E. Heinz, \textit{On the nonexistence of a surface of constant
mean curvature with finite area and prescribed rectifiable boundary. }Arch.
Rational Mech. Anal. \textbf{35} (1969), 249--252.

\bibitem {HoLiRo}D. Hoffman, J.H.S. de Lira, H. Rosenberg, \textit{Constant
mean curvature surfaces in $M^{2}\times\mathbf{R}$. }Trans. Amer. Math. Soc.
\textbf{358} (2006), no. 2, 491--507.

\bibitem {HPV}I. Holopainen, S. Pigola, G. Veronelli, \textit{Global
comparison principles for the }$p$\textit{-Laplace operator on Riemannian
manifolds.} Potential Anal.\textbf{ 34} (2011), no. 4, 371--384.

\bibitem {Ko}N.J. Korevaar, \textit{An easy proof of the interior gradient
bound for solutions to the prescribed mean curvature equation.} Nonlinear
functional analysis and its applications, Part 2 (Berkeley, Calif., 1983),
81--89, Proc. Sympos. Pure Math., 45, Part 2, Amer. Math. Soc., Providence,
RI, 1986.

\bibitem {ILPS-Killing}D. Impera, J.H.S. de Lira, S. Pigola, A.G. Setti,
\textit{Height estimates for Killing graphs.} In preparation.

\bibitem {IPS-preprint}D. Impera, S. Pigola, A.G. Setti, \textit{Nonlinear
potential theory on manifolds with boundary and applications to PDEs.} In preparation.

\bibitem {LW}P. Li, J. Wang,\textit{\ Finiteness of disjoint minimal graphs.}
Math. Research Letters \textbf{8} (2001), 771--777.

\bibitem {Lieberman- JMAA}G. M. Lieberman, \textit{Mixed boundary value
problems for elliptic and parabolic diffential equations of second order.} J.
Math. Anal. Appl. \textbf{113} (1986), 422--440.

\bibitem {LP}F. L\'{o}pez, J. P\'{e}rez, \textit{Parabolicity and Gauss map of
minimal surfaces.} Indiana Univ. Math. J. \textbf{52} (2003), 1017--1026.

\bibitem {LS}T. Lyons, D. Sullivan,\textit{\ Function theory, random paths and
covering spaces.} J. Diff. Geom. \textbf{18} (1984), 229--323.

\bibitem {Ne}R. Neel, \textit{Stochastic Methods for Minimal Surfaces.}
Contemp. Math., vol. \textbf{570}, Amer. Math. Soc., Providence, RI (2012), 111--136.

\bibitem {Ne2}R. Neel, \textit{Brownian motion and the parabolicity of minimal
graphs.} arXiv:0810.0669v1.

\bibitem {Pe}J. P\'{e}rez, \textit{Parabolicity and minimal surfaces.} Joint
work with F. J. L\'{o}pez. Clay Math. Proc., 2, Global theory of minimal
surfaces, 163--174, Amer. Math. Soc., Providence, RI, 2005.
http://www.ugr.es/\symbol{126}jperez/papers/ParabolicMSRI.pdf.

\bibitem {PRS-Pacific}S. Pigola, M. Rigoli, A.G. Setti, \textit{Some remarks
on the prescribed mean curvature equation on complete manifolds.} Pacific J.
Math. \textbf{206} (2002), 195--217.

\bibitem {PS-Fortaleza}S. Pigola, A.G. Setti, \textit{Global divergence
theorems in nonlinear PDEs and Geometry. }Lecture notes for the Summer School
in Differential Geometry held in Fortaleza, 2012.

\bibitem {PST}S. Pigola, A.G. Setti, M. Troyanov, \textit{The topology at
infinity of a manifold supporting an }$L^{q,p}$\textit{ Sobolev inequality.}
Submitted. First version: http://arxiv.org/abs/1007.1761. Second version: http://www.dfm.uninsubria.it/pigola/papers\&lecturenotes.

\bibitem {RS-Revista}M. Rigoli, A.G. Setti,\textit{\ Liouville type theorems
for }$\varphi$\textit{-subharmonic functions.} Rev. Mat. Iberoamer.
\textbf{17} (2001), 471--450.

\bibitem {RoRo}A. Ros, H. Rosenberg, \textit{Properly embedded surfaces with
constant mean curvature}. Amer. Jour. Math. \textbf{132} (2010), 1429-1443.

\bibitem {RSS-JDG}H. Rosenberg, F. Schulze, J. Spruck, \textit{The halfspace
property and entire positive minimal graphs.} To appear in Jour. Diff. Geom.

\bibitem {Sp}J. Spruck, \textit{Interior gradient estimates and existence
theorems for constant mean curvature graphs in }$M^{n}\times\mathbb{R}$.
\ Pure Appl. Math. Q. \textbf{3} (2007), no. 3, 785--800.

\bibitem {Wa}X.-J. Wang, \textit{Interior gradient estimates for mean
curvature equations.} \ Math. Z.\textbf{ 228} (1998), 73--81.
\end{thebibliography}
\end{document}